\documentclass[onecolumn,reqno]{amsart}
\usepackage{amsthm}
\usepackage{times,amssymb,amsmath,amsfonts}
\usepackage{graphicx,color}

\newtheorem{theorem}{Theorem}[section]
\newtheorem{lemma}[theorem]{Lemma}
\newtheorem{corollary}[theorem]{Corollary}
\newtheorem{proposition}[theorem]{Proposition}

\theoremstyle{definition}
\newtheorem{definition}[theorem]{Definition}

\newtheorem{remark}[theorem]{Remark}
\date{\today}

\begin{document}
		\title[Universal minima of discrete potentials for sharp spherical codes]{Universal minima of discrete potentials for sharp spherical codes}
	\author[P. G. Boyvalenkov]{P. G. Boyvalenkov$^1$} \thanks{$^1$The research of this author was supported, in part, 
by Ministry of Education and Science of Bulgaria under Grant no. DO1-387/18.12.2020 “National Centre for High-Performance and Distributed Computing”.}
\address{ Institute of Mathematics and Informatics, Bulgarian Academy of Sciences \\
8 G Bonchev Str., 1113  Sofia, Bulgaria}
\email{peter@math.bas.bg}

\author[P. D. Dragnev]{P. D. Dragnev$^2$}\thanks{$^2$The research of this author was supported, in part, by NSF grant DMS-1936543.}
\address{ Department of Mathematical Sciences, Purdue University \\
Fort Wayne, IN 46805, USA }
\email{dragnevp@pfw.edu}

\author{D. P. Hardin}
\address{ Center for Constructive Approximation, Department of Mathematics \\
Vanderbilt University, Nashville, TN 37240, USA }
\email{doug.hardin@vanderbilt.edu}

\author{E. B. Saff}
\address{ Center for Constructive Approximation, Department of Mathematics \\
Vanderbilt University, Nashville, TN 37240, USA }
\email{edward.b.saff@vanderbilt.edu}

\author[M. M. Stoyanova]{M. M. Stoyanova$^3$} \thanks{$^3$The research of this author was supported, in part, by Bulgarian NSF grant KP-06-N32/2-2019.}
\address{ Faculty of Mathematics and Informatics, Sofia University ``St. Kliment Ohridski"\\
5 James Bourchier Blvd., 1164 Sofia, Bulgaria}
\email{stoyanova@fmi.uni-sofia.bg}
	 \begin{abstract}  
This article is devoted to the study of discrete potentials on the sphere in $\mathbb{R}^n$ for sharp codes. We show that the potentials of most
of the known sharp codes attain the universal lower bounds for polarization for spherical $\tau$-designs previously derived by the authors, where ``universal'' is meant  in the sense of applying to a large class of potentials that includes absolutely monotone functions of inner products. We also extend our universal bounds to $T$-designs and the associated polynomial subspaces determined by the vanishing moments of spherical configurations and thus obtain the minima for the icosahedron, dodecahedron, and sharp codes coming from $E_8$ and the Leech lattice. For this purpose we investigate quadrature formulas for certain subspaces of Gegenbauer polynomials $P^{(n)}_j$ which we call PULB subspaces, particularly those having basis $\{P_j^{(n)}\}_{j=0}^{2k+2}\setminus \{P_{2k}^{(n)}\}.$ Furthermore, for potentials with $h^{(\tau+1)}<0$  we prove that the strong sharp codes and the antipodal sharp codes attain the universal bounds and their minima occur at points of the codes. The same phenomenon is established for the $600$-cell when the potential $h$ satisfies $h^{(i)}\geq 0$, $i=1,\dots,15$, and $h^{(16)}\leq 0.$

\end{abstract}

 		\maketitle


\section{Introduction}

Let $\mathbb{S}^{n-1}\subset \mathbb{R}^n$ denote the unit sphere. A collection of distinct points $C=\{x_1,x_2,\ldots,x_N\} \subset   \mathbb{S}^{n-1}$ is called {\em a spherical code}. 
For a function $h:[-1,1] \to [0,+\infty]$, continuous and finite on $[-1,1)$, we consider the {\em discrete $h$-potential associated with $C$}
\[ U_h(x,C):=\sum_{y\in C}h(x \cdot y),\]
where $x \in \mathbb{S}^{n-1}$ is arbitrary.  For a fixed cardinality of code and a fixed potential {\em polarization} is an optimization problem where we seek codes that attain the largest possible minimum (i.e. max-min) or the smallest maximum (min-max) on $\mathbb{S}^{n-1}$. In a recent work \cite{BDHSS_P} we derived polarization universal lower and upper bounds (PULB and PUUB)  for spherical codes and designs and related results. Here we continue this investigation by analyzing the minima of potentials generated by sharp codes, which we show attain our PULB. Furthermore, we generalize our previous work by developing what we call {\em second level polarization bounds}. 

Before introducing the polarization notions we remind the reader of the definition of spherical designs introduced in 1977 by Delsarte, Goethals and Seidel 
\cite{DGS} (for comprehensive surveys see \cite{Ban09,Ban17,BBIT}). 

\begin{definition} \label{def-designs-1} A spherical $\tau$-design $C
\subset \mathbb{S}^{n-1}$ is a finite subset of $\mathbb{S}^{n-1}$
such that
\[ \frac{1}{\mu(\mathbb{S}^{n-1})} \int_{\mathbb{S}^{n-1}} f(x) d\mu(x) =
                  \frac{1}{|C|} \sum_{x \in C} f(x) \]
($\mu$ is the surface area measure and $|C|$ denotes cardinality) holds for all polynomials $f(x)
= f(x_1,x_2,\ldots,x_n)$ of degree at most $\tau$ (i.e., the
average of $f$ over the set $C$ is equal to the average of $f$ over
$\mathbb{S}^{n-1}$). 
\end{definition}

This was extended to the concept of spherical $T$-designs by Delsarte and Seidel \cite{DS89} in 1989 (see also \cite[Section 6.1]{Ban17}). 
Given a spherical code $C  \subset   \mathbb{S}^{n-1}$, its {\em $i$-th moment}, $i\in \mathbb{N}$, is defined as
\begin{equation}\label{Moments}
M_i^n(C):=\sum_{x,y\in C} P_i^{(n)} (x\cdot y),
\end{equation}
where $P_i^{(n)} (t)$ are the Gegenbauer polynomials, normalized by $P_i^{(n)} (1)=1$. We remark that one of the alternative 
definitions for a spherical $\tau$-design is that all the moments $M_i^n(C)$, $i=1,\dots,\tau$, vanish.  

\begin{definition}\cite[Definition 6.1]{Ban17} \label{T-design}
Given an {\em index set} $T\subset \mathbb{N}$, we call a spherical code $C \subset \mathbb{S}^{n-1}$ an {\em $T$-design} if $ M_i^n(C)=0$ 
for every $i\in T$. 
\end{definition}

The polarization quantities we are going to investigate are defined as follows. Let 
\begin{equation} \label{max-min}
\mathcal{Q}_h(C):=\inf_{x \in \mathbb{S}^{n-1}} U_h(x,C), \ \ \mathcal{Q}_{h}(n,N):=\sup_{|C|=N, C\subset \mathbb{S}^{n-1}} \mathcal{Q}_h(C).
\end{equation}
Then the max-min and min-min polarization quantities for $T$-designs are defined as
\begin{equation} \label{Q-T-def}
\overline{\mathcal{Q}}_{h}(n ,N,T):=\sup \{ \mathcal{Q}_h(C)\, : \, |C|=N, \, C \ {\rm is\ a }\ T{\mbox -}{\rm design} \},
\end{equation}
and
\begin{equation} \label{Q-T-defUnder}
\underline{\mathcal Q}_h(n,N,T):=\inf \{\mathcal{Q}_h(C): |C|=N, \text{$C$ is a $T$-design}\}. 
\end{equation}
 Clearly, we have $\overline{\mathcal{Q}}_{h}(n ,N,T)\geq  \underline{\mathcal Q}_h(n,N,T)$. When $T=\{1,2,\ldots,\tau\}$ we write $\tau$ instead of $T$ in the above notations. 

When considering spherical codes, the {\em Gegenbauer polynomials} $P_i^{(n)} (t)=P_i^{(0,0)}(t)$ play an intricate role. We remind the reader that these are Jacobi\footnote{The Jacobi polynomials 
$P_i^{\alpha,\beta}(t)$ are orthogonal on $[-1,1]$ with respect to a weight function $(1-t)^\alpha (1+t)^\beta$.} polynomials $P_i^{\alpha,\beta}(t)$ 
with parameters
$\alpha=\beta=(n-3)/2$ normalized so that $P_i^{(n)} (1)=1$. Namely, the Gegenbauer polynomials $P_i^{(n)}(t)$ 
are orthogonal on $[-1,1]$ with orthogonality measure 
\[ d\mu_n(t):=\gamma_n (1-t^2)^{(n-3)/2}\, dt , \]
where the normalization constant $\gamma_n$ is chosen to make $\mu_n$ a probability measure. Recall that the {\em adjacent Gegenbauer polynomials} $P_i^{(a,b)} (t)$, $a,b\in\{0,1\}$, are Jacobi polynomials with 
$\alpha=a+(n-3)/2$ and $\beta=b+(n-3)/2$ similarly normalized by $P_i^{(a,b)} (1)=1$. Any real polynomial $f$ can be written as
\begin{equation} \label{geg-exp} 
f(t)=\sum_{i=0}^{\deg(f)} f_i P_i^{(n)}(t) 
\end{equation}
with Gegenbauer coefficients $f_i$ given by
\[f_i := \int_{-1}^1 f(t) P_i^{(n)}(t)\, d \mu_n(t)/\|P_i^{(n)}\|^2, \quad i=0,\dots, \deg(f). \]
We note that Gegenbauer polynomials are even/odd functions for even/odd $i$. 

We recall the polarization universal bounds derived in \cite{BDHSS_P}. The following equivalent definition of a spherical design 
facilitates our approach to the polarization problem (see \cite{DGS}, \cite[Equation (1.10)]{FL}).

\begin{definition} \label{def-des}  A code
$C \subset \mathbb{S}^{n-1}$ is a spherical $\tau$-design if and
only if for any point $x \in \mathbb{S}^{n-1}$ and any real 
polynomial $f(t)$ of degree at most $\tau$, the equality
\begin{equation}
\label{defin_f}
U_f(x,C)=\sum_{y \in C}f(x \cdot y) = f_0|C|
\end{equation}
holds, where $f_0=\int_{-1}^1 f(t) \, d \mu_n(t)$ is the constant coefficient in the Gegenbauer expansion  \eqref{geg-exp} of $f$. Similarly, $C$ is a spherical $T$-design if and only if \eqref{defin_f} holds for any $f\in \mathcal{P}_T$ (see Proposition \ref{prop23}), where
\begin{equation}\label{P_T}
\mathcal{P}_T:=\text{span } \{P_i^{(n)} (t):i\in T\cup \{0\}\}.
\end{equation}
\end{definition}

 For $x\in \mathbb{S}^{n-1}$ and a code $C \subset \mathbb{S}^{n-1}$, let $I(x,C):=\{x\cdot y :y\in C\}=\{u_i\}_{i=1}^\ell$ and let $r_i$ denote the relative    frequency of occurrence of $u_i$; i.e., $u_i = x\cdot y$ for $|C|r_i$ many distinct $y\in C$.  Observe that 
\begin{equation} \label{QRcode}
U_f(x,C)=\sum_{y \in C}f(x \cdot y) = |C|\sum_{i=1}^\ell r_i f(u_i).
\end{equation}
 
Note that \eqref{defin_f} asserts that the polarization $f$-potential for $C$ is constant on ${\mathbb S}^{n-1}$. This fact serves as the foundation in obtaining lower and upper linear programming (LP) bounds for polarization. In this article we are interested in lower bounds, so we  summarize the following PULB result from \cite{BDHSS_P}.

\begin{theorem}\label{PULB}  {\rm (\cite[Theorems 3.4 and 3.7, Corollaries 3.9 and 3.10]{BDHSS_P})} Suppose $C$ is   a spherical $\tau$-design of cardinality $N$ on $\mathbb{S}^{n-1}$, where $\tau=:2k-1+\epsilon$, $\epsilon\in \{0,1\}$, and that the potential $h$ is continuous on $[-1,1]$, finite on $(-1,1)$,
and has a derivative $h^{(2k+\epsilon)}$ of constant sign on $(-1,1)$.
Then 
\begin{equation}\label{PolarizationULB}
{Q}_h(C) \geq N \sum_{i\in I} \rho_i h(\alpha_i),
\end{equation}
where the index set $I$, the quadrature nodes $\{\alpha_i\}_{i\in I} $, and the positive weights $\{\rho_i\}_{i\in I}$   are determined as follows:
\begin{itemize}
\item[] (i) When $h^{(\tau+1)}(t)\geq 0$ on $(-1,1)$, then $I:=\{1-\epsilon,\dots,k\}$ and $\{\alpha_j\}_{j\in I}$ are the zeros of the (possibly adjacent) Gegenbauer polynomials $(1+t)^\epsilon P_k^{(0,\epsilon)}(t)$;
\item[] (ii) When $h^{(\tau+1)}(t)\leq 0$ on $(-1,1)$, then $I=\{0,1,\dots,k\}$ and $\{\alpha_j\}_{j\in I}$ are the zeros of the polynomials $(t-1)(t+1)^{1-\epsilon}P_{k-1+\epsilon}^{1,1-\epsilon}(t)$.
\end{itemize}
The weights  are positive, sum to 1, and are  given by
\begin{equation} \label{RhoWeights}
\rho_i:= \int_{-1}^1 \ell_i (t) \, d \mu_n(t)=\int_{-1}^1 \ell_i^2 (t) \, d \mu_n(t),
\end{equation} where
$\ell_i(t)$ denotes the Lagrange basic polynomials\footnote{ $\ell_i(\alpha_j)=\delta_{ij}$, the Kronecker delta.} associated with the nodes $\{ \alpha_j \}_{j=1-\epsilon}^k$.

Moreover, the bound \eqref{PolarizationULB} is the best that can be attained by linear programming via polynomials $f$
of degree at most $\tau$ for which $f\leq h$ on $[-1,1]$.
 
In addition, if a spherical $\tau$-design $C$ attains the bound \eqref{PolarizationULB}, then
there exists a point $\widetilde{x} \in \mathbb{S}^{n-1}$ such that the set $I(\widetilde{x},C)$ of all inner products between $\widetilde{x}$ and the points of $C$ coincides with the set
 $\{\alpha_i\}_{i\in I}$, and the multiplicities of these inner products are $\{N\rho_i \}_{i\in I}$, respectively. In particular, the numbers $N\rho_i$, $i\in I$,
are positive integers.
\end{theorem}

\begin{remark}\label{RemarkMinima}
We note that when a spherical $\tau$-design of cardinality $N$ on $\mathbb{S}^{n-1}$ exists, then Theorem~\ref{PULB} implies the following bounds \begin{equation}\label{PULB_N}
{\mathcal{Q}}_{h}(n,N) \geq  \underline{\mathcal Q}_h(n,N,\tau) \ge N\sum_{i\in I} \rho_i h(\alpha_i).
\end{equation}
\end{remark}

A by-product of Theorem \ref{PULB} is an alternative proof of the Fazekas-Levenshtein bound on the covering radius of spherical designs \cite[Theorem 2]{FL},
namely that the minimal (in terms of inner products) covering radius of a spherical $\tau$-design of cardinality $N$
(in fact, the Fazekas-Levenshtein bound does not depend on the cardinality) is at least as large as the 
largest quadrature node in \eqref{PolarizationULB}. In this regard, we showed in \cite{BDHSS_P} that the vertices of the cube on $\mathbb{S}^2$ and the $24$-cell on $\mathbb{S}^3$ both attain the bound \eqref{PolarizationULB} for potentials satisfying case (i), and as such the Fazekas-Levenshtein bound. The simplex and the cross-polytope on $ \mathbb{S}^{n-1}$ attain the PULB bound as well. In this article we show that the same is true for most of the known sharp codes.  However, prominent configurations, such as the icosahedron, dodecahedron, and the kissing configurations of  $E_8$, and the Leech lattice do not attain \eqref{PolarizationULB}. Thus, one of our goals is to enhance the PULB by considering polynomials $f$ with degree higher than $\tau$. As a consequence, we obtain stronger bounds, or second level PULB, which are attained by these codes.

In this paper we focus on {\em sharp codes} which are spherical $\tau$-designs which admit at most $[(\tau+1)/2]$ distinct 
inner products between their (distinct) points. All known sharp codes are listed in several papers, cf. \cite{Lev92,Lev,CK}, and no new sharp codes 
were discovered since 
1980's. We shall prove that all known sharp codes, except for the infinite family defined on the last row of Table 2, attain what we call first or second level PULB (the bounds \eqref{PolarizationULB} and
\eqref{level-2-bound}, respectively), and therefore the location of their minima are independent of the potential, and hence universal. The infinite family of the last row, along with the $600$-cell for potential satisfying case (i) of Theorem \ref{PULB}, will be considered in a future work.

This article is structured as follows. In Section \ref{prel} we introduce the needed preliminaries. Section \ref{pos-derivative} contains the first main result Theorem \ref{sharp-pos-level-1}, and is devoted to 
the analysis of the known sharp codes attaining the PULB \eqref{PolarizationULB} for potentials satisfying case (i) of Theorem \ref{PULB};
this is the first level PULB. 
We look into the structure of the corresponding codes and find suitable points to exhibit the universal minima.
In Section \ref{T-designs} we consider the general case of {\em spherical $T$-designs} for the case $T=\{1,2,\ldots,2k+2\} \setminus \{2k\}$
and present a Skip 1-Add 2 framework to enhance the max-min polarization bounds from Theorem \ref{PULB}.  We introduce the notion of PULB-space and prove that the polynomial subspace spanning the Gegenbauer polynomials of degree $i\in T\cup\{0\}$ is a PULB space in the second main result Theorem \ref{P_T PULB-space}. The 
second level PULB is given in Theorem \ref{level-2-T}.
Examples of codes attaining the new enhanced polarization bounds, namely the icosahedron, the dodecahedron, and the famous kissing configurations in $8$ and $24$ dimensions\footnote{M. Viazovska received a Fields medal in 2022: ``For the proof that the $E_8$ lattice provides the densest packing of identical spheres in 8 dimensions, and further contributions to related extremal problems and interpolation problems in Fourier analysis." - 
https://www.mathunion.org/imu-awards/fields-medal/fields-medals-2022}, are presented by Theorem 5.1 in Section \ref{sec-level2}. We enumerate the classification of the universal minima of the  known  sharp codes via the closest facet to such minima in Section \ref{classification}. Section \ref{neg-derivative} establishes that strongly sharp and antipodal sharp configurations attain the PULB \eqref{PolarizationULB} 
for potentials satisfying case (ii) of Theorem \ref{PULB}, as well as the minima for the $600$-cell is found for potentials $h$ that satisfy $h^{(i)}\geq 0$, $i=1,\dots,15$, and $h^{(16)}\leq 0$.

\section{Preliminaries} \label{prel}

\subsection{Quadrature rules and spherical designs}

\begin{definition}
\label{QuadRule}
Let $n$ be a positive integer and  $\Lambda$ a linear space of univariate polynomials that contains the constant polynomials.  We say that $\{\alpha_i\}_{i\in I}\subset [-1,1]$ and $\{\rho_i\}_{i\in I}\subset (0,1)$,  for some finite index set $I$,   form a {\em quadrature rule exact on $\Lambda$}   if
\begin{equation}\label{QR1}
f_0=\sum_{i\in I} \rho_i f(\alpha_i),\qquad f \in \Lambda.\end{equation} 
\end{definition}
If $C$ is a $T$-design, it follows from \eqref{defin_f} and \eqref{QRcode} that  $I(x,C)=\{\alpha_1,\ldots,\alpha_\ell\}$ with relative frequencies $\{\rho_1,\ldots, \rho_\ell\}$ form a quadrature rule that is exact on $\mathcal{P}_T$.

\subsection{The Delsarte-Goethals-Seidel bound, Levenshtein's $1/N$-quadrature rule and universal lower bounds (ULB) for energy}

The cardinality of spherical $\tau$-designs is bounded from below by the following Fisher-type bound, cf.
\cite[Theorems 5.11, 5.12]{DGS}. If $C\subset\mathbb{S}^{n-1}$ is a $\tau$-design, 
$\tau=2k-1+\epsilon$, $k \in \mathbb{N}$, $\epsilon \in \{0,1\}$, then 
\begin{equation}
\label{DGS-bound}
|C| \geq D(n,\tau):={n+k-2+\epsilon \choose n-1}+{n+k-2 \choose n-1}.
\end{equation}

The existence of $\tau$-designs on $\mathbb{S}^{n-1}$ with cardinality $N \geq D(n,\tau)$ is not guaranteed and Yudin \cite{Y-designs} showed
that the bound \eqref{DGS-bound} can be improved in some cases (see also \cite{BBD99}). On the other hand, Seymour and Zaslavsky \cite{SZ}
showed that there exists  $\tau$-designs on $\mathbb{S}^{n-1}$ with all large enough cardinalities, and Bondarenko, Radchenko, and Viazovska
\cite{BRV13,BRV15} solved a long-standing conjecture by proving that there exist spherical $\tau$-designs on $\mathbb{S}^{n-1}$ 
with all cardinalities $N \geq C_n \tau^{n-1}$, where $C_n$ depends only on the dimension $n$.

For every cardinality $N$, we let $\tau_{n,N}$ be the largest $\tau$ for which a spherical $\tau$-design of $N$ points on $\mathbb{S}^{n-1}$ exists. Then the definitions \eqref{max-min} and \eqref{Q-T-def} immediately imply the bound
\[ \mathcal{Q}_{h}(n,N)\geq \overline{\mathcal{Q}}_{h}(n,N,\tau_{n,N}). \]

Denote the maximal possible cardinality of a spherical code on $C\subset \mathbb{S}^{n-1}$ of prescribed maximal
inner product $s$ with
$$A(n,s):=\max\{|C| \colon C \subset \mathbb{S}^{n-1}, \langle x,y \rangle \leq s, \,   x\neq y \in C\}.$$
When introducing his bound on the quantity $A(n,s)$, Levenshtein utilized Gauss-type $1/N$-quadrature rules 
that we now briefly review (cf. \cite[Section 4]{Lev92}, \cite[Section 5]{Lev}). Given a
real number (possibly cardinality) $N$, there exists a unique $\tau=2k-1+\epsilon$, $\epsilon\in \{0,1\}$, such that  $ N \in (D(n,\tau),D(n,\tau+1)]$, where $D(n,\tau)$ are the Delsarte-Goethals-Seidel numbers \eqref{DGS-bound}. Let $\alpha_k=s$ be the maximal (unique in a certain interval) 
solution of the equation $N=L_{\tau}(n,t)$, where $L_{\tau}(n,t)$ is the Levenshtein function \cite[Section 5]{Lev}. Then there exist
uniquely determined quadrature nodes and nonnegative weights 
\[ -1 \leq \alpha_{1-\epsilon} < \cdots <\alpha_k < 1,\quad \rho_{1-\epsilon},\ldots,\rho_k \in \mathbb{R}^+, \]
such that the Radau/Lobatto $1/N$-quadrature holds:
\begin{equation} 
\label{defin_qf}
f_0= \frac{f(1)}{N}+\sum_{i=1-\epsilon}^{k} \rho_i f(\alpha_i), \ \ \mbox{ for all}\  f\in \mathcal{P}_{\tau},
\end{equation}
where $\mathcal{P}_\tau$ denotes the collection of polynomials of a single real variable of degree at most $\tau$.
When $\epsilon=1$, then $\alpha_0=-1$ and \eqref{defin_qf} is Lobatto quadrature, otherwise it is Radau quadrature.
The nodes $\alpha_i$, $i=1,\ldots,k$, are the roots of the equation
\[ P_k^{(\epsilon,1)}(t)P_{k-1}^{(\epsilon,1)}(\alpha_k) - P_{k}^{(\epsilon,1)}(\alpha_{k})P_{k-1}^{(\epsilon,1)}(t)=0,\label{op_eq1} \]
and the weights are found to determine the required accuracy.

\subsection{Sharp codes} We consider sharp codes, 
that is {\em spherical $\tau$-designs with $[(\tau+1)/2]$ distinct inner products among distinct points in the code}. Tables with all known sharp 
codes appeared earlier in the literature, as Levenshtein \cite[Table 9.1]{Lev92}, \cite[Table 6.2]{Lev} exhibited them as the all known codes 
attaining his upper bound on $A(n,s)$, and Cohn and Kumar \cite[Table 1]{CK} showed that they are universally optimal, i.e. they possess, 
for their dimension $n$ and cardinality $N$, the minimum possible $h$-energy for all absolutely monotone potentials $h$.   

All sharp codes attain the universal lower bound for $h$-energy as explained in the next subsection. In this paper 
we derive and explore one more aspect of the universality of the sharp codes by proving their optimality for Theorem 
\ref{PULB}, case (i). To this end, we exploit their combinatorial nature and analyze, case by case, how they attain the first and second level
polarization bounds \eqref{PolarizationULB} and \eqref{level-2-bound}, respectively.

We will use the notation $C=(n,N,\tau)$ for the sharp codes; that is, these are codes $C\subset \mathbb{S}^{n-1}$ of cardinality 
$|C|=N$ and design strength $\tau$ that have $[(\tau+1)/2]$ distinct inner products among distinct points in the code. If $\tau=2k$
(when $n \geq 3$, this may happen for $k=1$ and 2 only) then $C$ is called strongly sharp \cite{B}.

\subsection{Universal lower bound on energy of spherical codes}

As the sharp codes attain the analogous universal lower bound (ULB) for energy, it is beneficial to provide a comparison and the relevant context. In this subsection we briefly introduce the energy counterpart as developed in \cite{BDHSS}. 

It turns out that the Levenshtein's $1/N$-quadrature \eqref{defin_qf} plays an important role in bounding potential energy. Given a code $C\subset \mathbb{S}^{n-1}$ with cardinality $|C|=N$, the potential energy (or $h$-energy) of $C$ is defined as 
\[ E_h (n,C):=\sum_{x,y \in C,x\not= y} h(x\cdot y).\]
The optimization quantity 
\[ \mathcal{E}(n,N;h):=\inf_{|C|=N} E_h (n,C) \]
arises  in many areas such as crystallography, material science, information theory, etc. The following theorem holds.
\begin{theorem}[\cite{BDHSS}, Theorems 2.3 and 3.1]\label{THM_ULB}
Let $h$ be an absolutely monotone potential function on $[-1,1)$, $\{(\alpha_i, \rho_i)\}_{i=1}^{k}$ be the Levenshtein's $1/N$-quadrature rule 
(that is exact on $\mathcal{P}_\tau$). Then the universal lower bound holds
\begin{equation}\label{ULB1}
\mathcal{E}(n,N;h)\ge  N^2\sum_{i=1}^{k} \rho_i h(\alpha_i).
\end{equation}
\end{theorem}

\begin{center}
\begin{table}
\scalebox{0.75}{
\begin{tabular}{|c|c|c|c|c|c|c|}
\hline
dim & Cardinality & Strength & Energy (ULB bound)  \\ 
$n$ & $N$ & $\tau$ & $\mathcal{E}(n,N;h)/N$ \\
\noalign{\smallskip}\hline\noalign{\smallskip}
$2$ & $N=2k$ & $2k-1$ & $h(-1) + 2\sum\limits_{j=1}^{k-1} h(\cos(2j\pi/N))$  \\ 
\noalign{\smallskip}\hline\noalign{\smallskip}
$2$ & $N=2k+1$ & $2k$ & $2\sum\limits_{j=1}^k h(\cos(2j\pi/N))$  \\ 
\noalign{\smallskip}\hline\noalign{\smallskip}
$n$ & $N \leq n$ & 1 & $(N-1)h(-1/(N-1))$  \\ 
\noalign{\smallskip}\hline\noalign{\smallskip}
$n$ & $n+1$ & 2 & $nh(-1/n)$  \\ 
\noalign{\smallskip}\hline\noalign{\smallskip}
$n$ & $2n$ & 3 & $h(-1) + 2(n-1)h(0)$  \\ 
\noalign{\smallskip}\hline\noalign{\smallskip}
3 & 12 & 5 & $h(-1) + 5h(-1/\sqrt{5}) + 5h(1/\sqrt{5})$  \\ 
\noalign{\smallskip}\hline\noalign{\smallskip}
5 & 16 & 3 & $5h(-3/5) + 10h(1/5)$  \\ 
\noalign{\smallskip}\hline\noalign{\smallskip}
6 & 27 & 4 & $10h(-1/2) + 16h(1/4)$  \\ 
\noalign{\smallskip}\hline\noalign{\smallskip}
7 & 56 & 5 & $h(-1) + 27h(-1/3) + 27h(1/3)$  \\ 
\noalign{\smallskip}\hline\noalign{\smallskip}
8 & 240 & 7 & $h(-1) + 56h(-1/2) + 126h(0) + 56h(1/2)$  \\ 
\noalign{\smallskip}\hline\noalign{\smallskip}
21 & 112 & 3 & $30h(-1/3) + 81h(1/9)$  \\ 
\noalign{\smallskip}\hline\noalign{\smallskip}
21 & 162 & 3 & $56h(-2/7) + 105h(1/7)$  \\ 
\noalign{\smallskip}\hline\noalign{\smallskip}
22 & 100 & 3 & $22h(-4/11) + 77h(1/11)$  \\ 
\noalign{\smallskip}\hline\noalign{\smallskip}
22 & 275 & 4 & $112h(-1/4) + 162h(1/6)$  \\ 
\noalign{\smallskip}\hline\noalign{\smallskip}
22 & 891 & 5 & $42h(-1/2) + 512h(-1/8) + 336h(1/4)$  \\ 
\noalign{\smallskip}\hline\noalign{\smallskip}
23 & 552 & 5 & $h(-1) + 275h(-1/5) + 275h(1/5)$  \\ 
\noalign{\smallskip}\hline\noalign{\smallskip}
23 & 4600 & 7 & $h(-1) + 891h(-1/3) + 2816h(0) + 891h(1/3)$  \\ 
\noalign{\smallskip}\hline\noalign{\smallskip}
24 & 196560 & 11 & $h(-1) + 4600h(-1/2) + 47104h(-1/4) + 93150h(0) + 47104h(1/4) + 4600h(1/2)$  \\
\noalign{\smallskip}\hline\noalign{\smallskip}
$q(q^3+1)/(q+1)$ & $(q^3+1)(q+1)$ & 3 & $q(q^2+1)h(-1/q)+q^4h(1/q^2)$ \\
$q$ -- a prime power &&& \\
\hline
\end{tabular}
}
\bigskip
\caption{The energy ULB bound for sharp codes}
\label{EnergyULB_Table}
\end{table}
\end{center}

Table 1 displays the scaled energy $\mathcal{E}(n,N;h)/N$ of all sharp codes, which coincides with the ULB \eqref{ULB1} divided by the cardinality $N$. 
Namely, the inner products are given by $\alpha_i$ (the Levenshtein nodes) and the numbers of occurrences of these inner products are the positive
integers $N\rho_i$ (the Levenshtein weights). We note that these spectral parameters (inner products and number of their occurrences) determine uniquely the sharp codes of dimension $n$ and cardinality $N$ (see \cite[Appendix A]{CK}, \cite{BS81,CK-2}; the uniqueness of smaller codes follows from the uniqueness of corresponding strongly regular graphs). Therefore, we will use the information from Table 1 in order to identify the sharp codes via their distance distribution as shown by the coefficients in front
of the values of the potential $h$. 

\subsection{Strongly regular graphs} Our analysis of the universal minima of the the sharp codes will utilize the notion of {\em strongly regular graphs} srg$(v,\ell,\lambda,\mu)$, that is graphs with $v$ vertices, $\ell$ edges stemming out of each vertex, $\lambda$ neighbors to any two adjacent vertices, and $\mu$ neighbors to any non-adjacent pair. Our main reference to the properties of these fascinating mathematical structures will be the recent comprehensive book of Brouwer and Van Maldeghem \cite{BrSRG}. The relation between the strongly regular graphs and sharp codes was described by Cameron, Goethals and Seidel in \cite{CGS}. 

\section{Universal minima for sharp codes: $h^{(\tau+1)}(t)\geq 0$ case}\label{pos-derivative}

In this section we shall provide a detailed analysis on the relation between our PULB \eqref{PolarizationULB} in the case (i) of Theorem \ref{PULB}
and the sharp codes in the case $T=\{1,2,\ldots,\tau\}$. 

Table 2 lists the parameters of the PULB \eqref{PolarizationULB} when $h^{(\tau+1)} \geq 0$,  the case (i) of Theorem \ref{PULB}, for the sharp codes $C=(n,N,\tau)$. 
As we shall see in this section, most sharp codes attain \eqref{PolarizationULB} with point(s) $\widetilde{x}$ and distance distributions exactly as shown in Table 2. 
Yet there are some notable exceptions, such as the icosahedron, and the kissing configurations of the $E_8$ lattice and 
the Leech lattice. Indeed, their coefficients $N\rho_i$ in the table are non-integer, so they cannot attain \eqref{PolarizationULB}. However, these exceptions 
serve as a motivation to develop the framework for an enhanced PULB for $T$-designs which will be considered in next sections.  

 We use the data from Table 2 to identify the structure of the sharp codes with respect to the sought point $\widetilde{x}$. Consequently,
we are able to present  $\widetilde{x}$ explicitly in all cases under consideration. 

\begin{center}
\begin{table}
\scalebox{0.5}{
\begin{tabular}{|c|c|c|c|c|c|c|}
\hline
dim & Cardinality & Strength & Polarization (PULB bound) \\ 
$n$ & $N$ & $\tau$ & $h^{(\tau+1)} \geq 0$ \\
\noalign{\smallskip}\hline\noalign{\smallskip}
$2$ & $N=2k$ & $2k-1$ & $2\sum\limits_{j=1}^k h(\cos((2j-1)\pi/N))$ \\
\noalign{\smallskip}\hline\noalign{\smallskip}
$2$ & $N=2k+1$ & $2k$ & $h(-1) + 2\sum\limits_{j=1}^k h(\cos(2j\pi/N))$ \\
\noalign{\smallskip}\hline\noalign{\smallskip}
$n$ & $N \leq n$ & 1 & $Nh(0)$ \\ 
\noalign{\smallskip}\hline\noalign{\smallskip}
$n$ & $n+1$ & 2 & $h(-1) + nh(1/n)$ \\ 
\noalign{\smallskip}\hline\noalign{\smallskip}
$n$ & $2n$ & 3 & $nh(-1/\sqrt{n}) + nh(1/\sqrt{n})$ \\ 
\noalign{\smallskip}\hline\noalign{\smallskip}
$3^{*}$ & 12 & 5 & $4[(5/6)h(-\sqrt{3/5}) + (4/3)h(0) + (5/6)h(\sqrt{3/5})]$ \\ 
\noalign{\smallskip}\hline\noalign{\smallskip}
5 & 16 & 3 & $8h(-1/\sqrt{5}) + 8h(1/\sqrt{5})$ \\ 
\noalign{\smallskip}\hline\noalign{\smallskip}
6 & 27 & 4 & $h(-1) + 16h(-1/4) + 10h(1/2)$ \\ 
\noalign{\smallskip}\hline\noalign{\smallskip}
7 & 56 & 5 & $12h(-1/\sqrt{3}) + 32h(0) + 12h(1/\sqrt{3})$ \\ 
\noalign{\smallskip}\hline\noalign{\smallskip}
$8^{*}$ & 240 & 7 & $240[((6 -\sqrt{15})/24)h(-\sqrt{25+5\sqrt{15}}/10) +((6 +\sqrt{15})/24)h(-\sqrt{25-5\sqrt{15}}/10) +((6+\sqrt{15})/24) h(\sqrt{25-5\sqrt{15}}/10) +((6 -\sqrt{15})/24)h(\sqrt{25+5\sqrt{15}}/10)]$ \\ 
\noalign{\smallskip}\hline\noalign{\smallskip}
21 & 112 & 3 & $56h(-1/\sqrt{21}) + 56h(1/\sqrt{21})$ \\ 
\noalign{\smallskip}\hline\noalign{\smallskip}
21 & 162 & 3 & $81h(-1/\sqrt{21}) + 81h(1/\sqrt{21})$ \\ 
\noalign{\smallskip}\hline\noalign{\smallskip}
22 & 100 & 3 & $50h(-1/\sqrt{22}) + 50h(1/\sqrt{22})$ \\ 
\noalign{\smallskip}\hline\noalign{\smallskip}
22 & 275 & 4 & $h(-1) + 162h(-1/6) + 112h(1/4)$ \\ 
\noalign{\smallskip}\hline\noalign{\smallskip}
22 & 891 & 5 & $162h(-1/\sqrt{8}) + 567h(0) + 162h(1/\sqrt{8})$ \\ 
\noalign{\smallskip}\hline\noalign{\smallskip}
23 & 552 & 5 & $100h(-\sqrt{3}/5) + 352h(0) + 100h(\sqrt{3}/5)$ \\ 
\noalign{\smallskip}\hline\noalign{\smallskip}
23 & 4600 & 7 & $275h(-\sqrt{5}/5) + 2025h(-\sqrt{5}/15) + 2025h(\sqrt{5}/15) + 275h(\sqrt{5}/5)$ \\ 
\noalign{\smallskip}\hline\noalign{\smallskip}
$24^{*}$ & 196560 & 11 & $(1207.983)h(-0.577) + (21794.872)h(-0.349) + (75277.144)h(-0.117) + (75277.144)h(0.117) + (21794.872)h(0.349) + (1207.983)h(0.577)$ \\
\noalign{\smallskip}\hline\noalign{\smallskip}
$\frac{q(q^3+1)}{q+1}^{*}$ & $(q^3+1)(q+1)$ & 3 & $\frac{(q^3+1)(q+1)}{2}[h(-\sqrt{\frac{1}{q^3-q^2+q}}) + h(\sqrt{\frac{1}{q^3-q^2+q}})]$ \\
   &      &      &  \\
\hline
\end{tabular}
}
\bigskip
\caption{Universal minima for sharp codes, $h^{(\tau+1)}(t)\geq 0$ case. Codes that do not attain the PULB \eqref{PolarizationULB} (the icosahedron, the kissing configurations of $E_8$ and the Leech lattice) are indicated with *.}
\end{table}
\end{center}

Regarding the infinite family of sharp codes on the last row of Table 2, we comment that for $q=2$ and $q=3$ these are the rows with $N=27$ and $N=112$. For $q$ a power of a prime number, a strongly regular graph srg$((q^3+1)(q+1) ,q(q^2+1),q-1,q^2+1)$ exists (see \cite[Chapter 3]{BrSRG}, \cite[Chapter 9]{EZ}). For  $q=2^m$, $m\geq 2$ the coefficients are not integer and so the first-level bound is not attained. This family, along with the $600$-cell will be considered in detail in a future work.

The following theorem outlines the main result of this section with the specific details for each code following in separate subsections. More detailed classification of the points where the universal minima are attained is summarized in Section \ref{classification} and Theorem \ref{6_1}. 

\begin{theorem} \label{sharp-pos-level-1}
All unmarked sharp codes from Table 2 attain the bound \eqref{PolarizationULB}  when $h^{(\tau+1)} \geq 0$, case (i) of Theorem \ref{PULB}, with inner products and distance distributions following the data from that table 
and points $\widetilde{x}$ as described below. 
\end{theorem}


\subsection{Universal minima for $N$-gons, simplexes, and cross-polytopes} In this subsection we summarize what is already known in the literature about polarization of sharp codes, namely that the codes in the first five lines of Table 2 all attain our PULB \eqref{PolarizationULB}.

The spherical codes for which $N\leq n$ are considered in \cite[Proposition 14.2.1]{BHS}, where it is shown that the necessary and sufficient condition for a spherical code to attain its max-min polarization among all configurations of cardinality $N$ is that the code is a $1$-design  (see also \cite[Example 6.1]{BDHSS_P}). Since the sharp codes in this case are embedded simplexes with design strength one, they attain our PULB \eqref{PolarizationULB}.

The max-min polarization optimality of the regular simplex was only recently established by Su \cite{Su} for the case $n=3$ and for general $n$ by Borodachov in \cite{B2} \textcolor{black}{(see also \cite{NR2})}. Using our PULB and PUUB results, an alternative proof was provided in \cite[Example 6.2]{BDHSS_P}.

The next class of sharp codes, the cross-polytopes $C_{2n}$ of $2n$ points on $\mathbb{S}^{n-1}$, also attain \eqref{PolarizationULB}. This was shown in \cite{NR2} for the case of power potentials and in \cite[Example 6.5]{BDHSS_P} for general potentials satisfying $h^{(\tau+1)}(t)\geq 0$. 

At the end of this subsection we consider the code from the first two lines of Table 2, the regular $N$-gon $C_N$. In this case the configuration is an optimal max-min polarization code on the unit circle (see \cite[Section 14.3]{BHS}). We shall derive this optimality utilizing our PULB \eqref{PolarizationULB}. 
Consider when $N$ is even, i.e. $N=2k$. In this case the regular $N$-gon is a $(2k-1)$-design on the circle $\mathbb{S}^1$ and Theorem \ref{PULB}, Case (i) with $\epsilon=0$ holds. The polynomials $P_k^{(0,0)}(t)=\cos(k\arccos t)$ are the Chebyshev polynomials with nodes $\alpha_i =\cos((2i-1)\pi/(2k))$ and equal weights $\rho_i=1/k$, $i=1,\dots,k$. Clearly, a mid-point $\widetilde{x}$ of any arc induced by the points $\alpha_i$ will satisfy the PULB \eqref{PolarizationULB} with equality; i.e., 
\[ \mathcal{Q}_h (C_N) =U_h (\widetilde{x},C_N)= \sum_{i=1}^k 2h(\cos((2i-1)\pi/N)\quad (=\underline{\mathcal Q}_h(n,N,\tau)). \]
The case of odd $N$ is similar.

\subsection{The sharp code $(5,16,3)$} As noted by Cohn-Kumar in \cite{CK} the codes in rows 7-9 of Table 2 are obtained as {\em kissing spherical caps} configurations of the  codes below them. Namely, fixing a point in the $(8,240,7)$-code derived from the $E_8$ lattice, the $56$ ``nearest neighbors" that have the largest inner product $1/2$ (see Table 1) form a scaled version of the $(7,56,5)$-code. Similarly, the $27$ ``nearest neighbors" 
of a fixed point on $(7,56,5)$-code form a scaled version of the $(6,27,4)$-code and the $16$ ``nearest neighbors" of a fixed point on $(6,27,4)$-code 
form a scaled version of the $(5,16,3)$-code. 

Table 1 reveals that if we fix any point in the $(5,16,3)$-code, there are $10$ points at inner product $1/5$ and $5$ points at an inner product of $-3/5$ with the fixed point. This property defines uniquely the code.

Upon inspection of Table 2, we conclude that should $C_{16}:=(5,16,3)$ attain \eqref{PolarizationULB}, there has to be a point $\widetilde{x}\in \mathbb{S}^4$ such that $C_{16}$ splits into two subsets $A:=\{a_i\}_{i=1}^8$ and $B:=\{b_j\}_{j=1}^8$, such that $\widetilde{x}\cdot a_i =1/\sqrt{5}$, $i=1,\dots, 8$, and $\widetilde{x}\cdot b_j=-1/\sqrt{5}$, $j=1,\dots, 8$.

Without loss of generality we may assume $\widetilde{x}=(0,0,0,0,1)$. Define the points $a_i$ to be the eight permutations 
\[A:=\{ (\pm \frac{2}{\sqrt{5}},0,0,0,\frac{1}{\sqrt{5}})\}=\{a_i\}_{i=1}^8 ,\] 
where the last coordinate stays fixed and the other non-zero coordinate switches  positions.  
We next define the points 
\[ B:=\{ (\pm \frac{1}{\sqrt{5}},\pm \frac{1}{\sqrt{5}},\pm \frac{1}{\sqrt{5}},\pm \frac{1}{\sqrt{5}},-\frac{1}{\sqrt{5}})\}= \{b_i\}_{i=1}^8, \]
where the number of negative signs is even. 
Obviously, $\widetilde{x} \cdot a_i=1/\sqrt{5}=-\widetilde{x} \cdot b_i$, i.e. the point $\widetilde{x}$ has the required properties.

We next verify that the constructed code has the same inner product distribution as $C_{16}$. Fix $a_1=(2/\sqrt{5},0,0,0,1/\sqrt{5})$. 
Then there is $a_1\cdot a_i$ takes on the values of $-3/5$ and $1/5$ of multiplicities $1$ and $6$ respectively. 
Similarly, we can compute $a_1\cdot b_j$, which takes on the values of $-3/5$ and $1/5$ of multiplicities $4$ and $4$ respectively. Thus, the total inner product distribution is the same. The symmetry implies that the same is true for all points in the code.

\subsection{The sharp codes $(6,27,4)$ and $(22,275,4)$ } In these cases the determination of $\widetilde{x}$ is straightforward. 
Table 1 reveals that a fixed point $y$ in $(6,27,4)$ has inner products with the other points in the code $-1/2$ and $1/4$ with multiplicity $10$ and $16$ respectively. Then the antipodal point $\widetilde{x}=-y$ has the required inner products with matching multiplicity for the quadrature rule from Table 2. The same argument applies to the second code $(22,275,4)$. We note that both codes are strongly sharp and not antipodal.

\subsection{The sharp code $(7,56,5)$ -- equiangular lines} This code is obtained as the intersection of a famous set of 
$28$ equiangular lines in $\mathbb{R}^7$ with the unit sphere $\mathbb{S}^6$ (see \cite[Chapter 11]{GR}). In this case our considerations go along the lines of Subsection 3.2. Let us select a point $\widetilde{x}=(0,0,0,0,0,0,1)$. Informed by Table 2 we seek to split $C_{56}:=(7,56,5)$ into three subsets 
\[ C_{56}=A\cup E\cup B, \]
where $A$ and $B$ will have altitude (last coordinate) $1/\sqrt{3}$ and $-1/\sqrt{3}$, respectively, and cardinalities $|A|=|B|=12$ and where $E$ is on the Equator (last coordinate $0$) and $|E|=32$. Let $A$ and $B$ be scaled down cross-polytopes at altitudes $1/\sqrt{3}$ and $-1/\sqrt{3}$, namely consider the twelve possible permutations for each of the sets (last coordinate fixed)
\[A:=\left\{ \left(\pm \sqrt{\frac{2}{3}},0,0,0,0,0,\frac{1}{\sqrt{3}}\right) \right\} =\left\{a_i\right\}_{i=1}^{12},\quad B:=\left\{ \left(\pm \sqrt{\frac{2}{3}},0,0,0,0,0,-\frac{1}{\sqrt{3}}\right) \right\} =\left\{b_j\right\}_{j=1}^{12}.\] 

Define the set $E$ to consist of the $32$ permutations of the type
\[ E:=\left\{ \left(\pm \frac{1}{\sqrt{6}},\pm \frac{1}{\sqrt{6}},\pm \frac{1}{\sqrt{6}},\pm \frac{1}{\sqrt{6}},\pm\frac{1}{\sqrt{6}},\pm \frac{1}{\sqrt{6}},0\right)\right\}=\{ e_\ell \}_{\ell=1}^{32},\]
where we either have all positive (1) or all negative (1) signs, or two negative (15) or four negative (15) signs. Fix $a_1:=(\sqrt{2/3},0,0,0,0,0,\sqrt{1/3})$. Then: 
\begin{itemize}
\item $a_1\cdot a_i$ has values $-1$, $-1/3$ and $1/3$ of multiplicities $0$, $1$ and $10$ respectively; 
\item $a_1\cdot b_j$ has values $-1$, $-1/3$ and $1/3$ of multiplicities $1$, $10$, and $1$ respectively;
\item  $a_1\cdot e_\ell$ has values $-1$, $-1/3$ and $1/3$ of multiplicities $0$, $16$, and $16$ respectively,
\end{itemize}
thus $I(a_1,C_{56})=\{-1,-1/3,1/3\}$ with multiplicities $1,27,27$, respectively. It is similarly straightforward to see that all other points of $C$ have 
the same distance distribution, so the uniqueness implies that this sharp code also attains the PULB bound \eqref{PolarizationULB}.

\subsection{The sharp code $(22,100,3)$ from the binary Golay code and the Higman-Sims graph} The next sharp code has a remarkable connection to the {\em Higman-Sims graph} discovered by Dale Mesner in 1956 \cite{Me}  and re-discovered in a different setting by Higman and Sims in 1968 \cite{HiSi}
(see Figure 1). This is a unique strongly regular graph srg$(100,22,0,6)$ with $100$ vertices and $1100$ edges \cite{BrSRG}. Each vertex has degree $22$. Any two adjacent vertices share no common neighbor (triangle-free graph) 
and any two non-adjacent vertices have exactly $6$ common neighbors. We shall explicitly construct the 
$C_{100}:=(22,100,3)$ code and find all points $\widetilde{x}\in \mathbb{S}^{21}$ whose inner products to half the points of the code are $-1/\sqrt{22}$ and to the other half are $1/\sqrt{22}$, exactly as the row of $C_{100}$ in Table 2 suggests.

The description of the construction of $C_{100}$ starts with the {\em binary Golay code} $\mathcal{C}_{23}$, ``probably the most important of all codes, for both practical and theoretical reasons" 
\cite{MWS}. We describe the necessary properties of the Golay codes following the book \cite{MWS} (Chapter 2, \textsection 2.6, Chapter 20). 
The code  $\mathcal{C}_{23}$ is linear (i.e., a subspace of $\{0,1\}^{23}$), has length 23, dimension $12$ (i.e., cardinality $2^{12}=4096$) 
and minimum distance $7$ (i.e. every two distinct codewords differ in at least 7 positions) \cite[Theorems 6 and 7]{MWS}. The number of
the codewords of minimum weight (i.e., the codewords with exactly 7 ones) is 253. 

\begin{figure}[htbp]
\centering
\includegraphics[width=3 in]{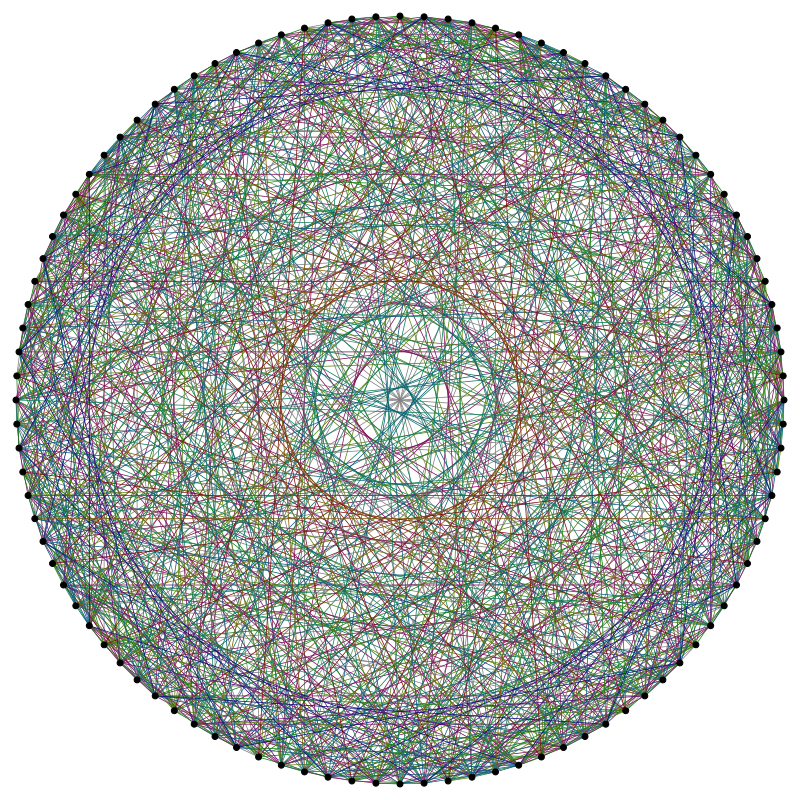}
\caption{Higman-Sims graph (By Claudio Rocchini - Own work, CC BY 3.0, https://commons.wikimedia.org/w/index.php?curid=4242731)}
\label{fig:1}
\end{figure}

The starting point of our construction will be the $253$ binary codewords of length $23$ from $\mathcal{C}_{23}$
that have weight $7$. There are $77$ out of these $253$ that start with $1$. Fix these $77$ binary codewords and 
eliminate their first digit $1$. Thus, we are left with $77$ binary words $\alpha_i$, $i=1,\dots,77$, of 
length $22$ and weight $6$, a set we denote by $\mathcal{A}$. Moreover, any two distinct words from $\mathcal{A}$
may intersect in common ones at $0$ or $2$ positions only. We remark that the remaining $253-77=176$ points of minimum weight 
in $\mathcal{C}_{23}$ still have to play important roles -- they will define the $2 \cdot 176=352$ points on $\mathbb{S}^{21}$, 
where the extrema of the polarization is attained. 


After this preparation we are ready to explicitly describe the construction of the spherical code $C_{100}$. 
Let us fix a point $c\in C_{100}$. Table 1 reveals that we need to find $77$ points $A:=\{a_i\}_{i=1}^{77}$ in $C_{100}$ that have inner product $1/11$ and $22$ points $B:=\{b_j\}_{j=1}^{22}$ that have inner product $-4/11$, i.e. $c\cdot a_i=1/11$, $i=1,\dots,77$ and $c\cdot b_j=-4/11$, $j=1,\dots,22$. Consider the points of $C_{100}$ as vertices of a graph and adjoin two vertices with an edge when the inner product between the corresponding points is $-4/11$. Now 
the construction of $C_{100}$ will follow closely the $1+22+77$ construction of the Higman-Sims graph  \cite{BrSRG}.

We now determine the set $A\subset \mathbb{S}^{21}$ by substituting $1$'s and $0$'s in $\alpha_i$ with $x$ and $y$
 to form the two letter $22$-length words $a_i$. We require that the inner product between $a_i$ and $a_j$ be $-4/11$ 
when they have no intersection and $1/11$ when they intersect at $2$ positions. Thus, we obtain the following system
\begin{equation}\label{xy_system}
\begin{split}
12 xy+10y^2&= - \frac{4}{11}\\
2x^2 + 8xy+ 12 y^2&=  \frac{1}{11}
\end{split}
\end{equation}
(observe that $6x^2+16y^2=1$, so $a_i\in\mathbb{S}^{21}$). Solving the system \eqref{xy_system} we obtain
\[ x=\frac{8\sqrt{5}-1}{11\sqrt{22}} \approx 0.32733, \quad \quad  y=-\frac{3\sqrt{5}+1}{11\sqrt{22}} \approx -0.1494. \]

We next construct the points of $B$ to be $22$-length words over the alphabet $\{z,u\}$ with exactly one $z$ and $21$ $u$'s. Requiring the inner product among points of $B$ to be $1/11$ and that they lie on $\mathbb{S}^{21}$ leads to the system
\begin{equation}\label{zu_system}
\begin{split}
2zu+20u^2&= \frac{1}{11}\\
z^2 + 21u^2&= 1
\end{split}
\end{equation}
with solution
\[ z=\frac{4-21\sqrt{5}}{11\sqrt{22}}, \quad \quad  u=\frac{4+\sqrt{5}}{11\sqrt{22}}. \]
We finally let
\[c:=(-\frac{1}{\sqrt{22}},-\frac{1}{\sqrt{22}}, \dots, -\frac{1}{\sqrt{22}})\in \mathbb{S}^{21}.\]

A straightforward verification yields that the constructed code $C_{100}$ has the same distance distribution as the unique $(22,100,5)$ sharp code 
(cf. the data in Table 1). 

We next construct the points of extremal polarization. For this purpose we consider the $176=253-77$ words of weight $7$ from $\mathcal{C}_{23}$
beginning with 0 and erase that 0. We denote the set of the remaining $22$-length binary codewords of weight $7$ as $\mathcal{B}$. 
These codewords are used to split the Higman-Sims graph into two copies ($1+7+42$ and $15+35$) of the {\em Hoffman-Singleton graph} (see \cite{HS60,Br}) as follows. Given a fixed  codeword of $\mathcal{B}$, form $\widetilde{x}\in \mathbb{S}^{21}$ 
that has $7$ $a$'s at the positions of the 7 
ones and $15$ $b$'s at the positions of the $15$ zeros. Let $C_1$ be the set of $50=1+7+42$ points of $C$ that include the point $c$, the $7$ points in $B$ that have $z$ in common position with an $a$ from  $\widetilde{x}$, and the $42$ points of $A$ that have one $x$ and $a$ in common position. Then the set 
$C_2$ of $50=15+35$ points is made of the other $15$ points from $B$ and the $35$ points from $A$ that have three $x$'s and $a$'s in common positions. 

Utilizing the conditions that $\widetilde{x}\cdot c =-1/\sqrt{22}$ and $\widetilde{x}\in \mathbb{S}^{21}$ we get the system
\[ \begin{split}
7a+15b&= 1\\
7a^2 + 15b^2&= 1,
\end{split} \]
which yields that
\[ a=\frac{5+15\sqrt{5}}{110}, \quad \quad  b=\frac{5-7\sqrt{5}}{110}. \]
 
 We verify directly the following inner product equalities
 \[ a(x+6y)+b(5x+10y)=-\frac{1}{\sqrt{22}}, \quad a(3x+4y) +b(3x+12y)= \frac{1}{\sqrt{22}},\]
 and 
\[ a(z+6u)+15bu=-\frac{1}{\sqrt{22}},\quad 7au+b(z+14u)= \frac{1}{\sqrt{22}}.\]

This implies that the sharp code $C_{100}=(22,100,3)$ attains the PULB \eqref{PolarizationULB}. \hfill $\Box$

\subsection{The sharp code $(22,275,4)$ and the MacLaughlin graph} \label{275} Even though we already have determined 
in subsection 3.3 that the sharp code $C_{275}:=(22,275,4)$ attains the PULB \eqref{PolarizationULB}, we shall explicitly 
construct this code as it will facilitate our analysis of the sharp codes $(21,112,3)$, $(21,162,3)$, and $(23, 552, 5)$. The construction is based on the MacLaughlin graph \cite{MacL}, a strongly regular graph srg$(275,112,30,56)$, with $275$ vertices and $15,400$ edges (see Figure 2). Each vertex has degree $112$, any two adjacent vertices share $30$ common neighbors and any two non-adjacent vertices have exactly $56$ common neighbors. The complementary graph is also a strongly regular graph srg$(275,162,105,81)$ with degree of the vertices $162$, adjoint vertices sharing $105$ common neighbors, while disjoint vertices sharing $81$ neighbors  \cite{BrSRG}. We shall explicitly construct the 
$C_{275}$ sharp code and find all points $\widetilde{x}\in \mathbb{S}^{21}$ whose inner products and frequency match the PULB data in Table 2.

\begin{figure}[htbp]
\centering
\includegraphics[width=3 in]{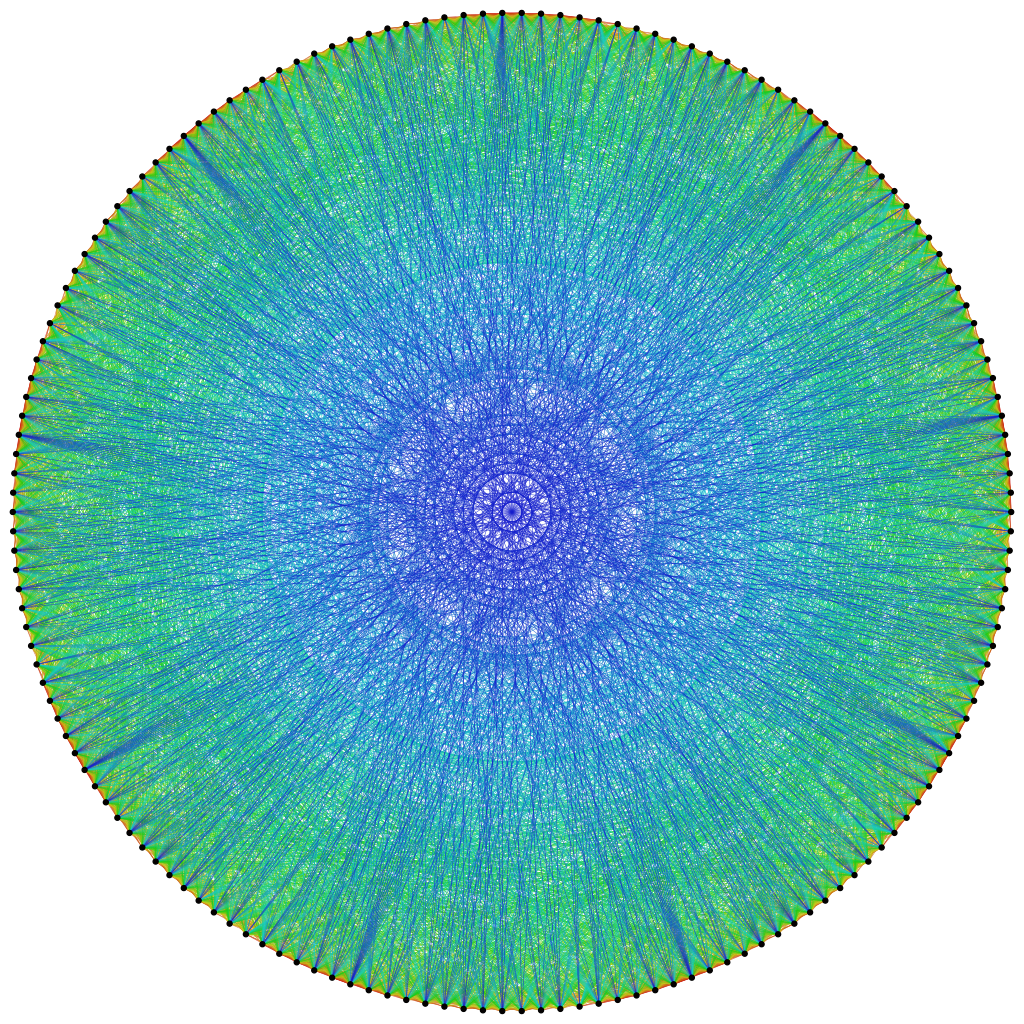}
\caption{McLaughlin graph (By Claudio Rocchini - Own work, CC BY-SA 3.0, https://commons.wikimedia.org/w/index.php?curid=10840948)}
\label{fig:2}
\end{figure}

Our starting point is again the $253$ binary codewords of length $23$ from the Golay code $\mathcal{C}_{23}$ that have weight $7$, which we shall call {\em blocks}.  Denote the collection of $77$ codewords that end with $1$ in the $23$-rd position with $A$ and the $176$ codewords that end in $0$ with $B$. Introduce the collection 
\[P:=\{[1,0,\dots,0],[0,1,\dots,0],\dots, [0,0,\dots,1]\} \]
of $23$ additional codewords of length $23$, which we call {\em points}. The combined $253+22+1=276$ codewords of points and blocks will be regarded as vertices of a graph. A point $p_i$ and a block $b_j$ will be adjacent if they don't share a $1$ on the same position, or $b_j[i]=0$. Two blocks may share $1$'s in only $1$ or $3$ positions (recall that the octads of  $\mathcal{C}_{23}$ intersect in $0$, $2$, or $4$ positions and we eliminated a column of $1$'s), they will be adjacent when they share $1$ in one position only. 

Let us fix the vertex (point) $p_{23}$. It will be connected with the $176$ blocks in $B$ and no other vertices in $P$ or $A$. Any other vertex in $P$ is connected to $120$ vertices in $B$, no vertices in $P$ and $56$ vertices in $A$. Any vertex in $A$ is adjacent to $96$ vertices in $B$, $16$ vertices in $P$, and $16$ in $A$. Finally, any vertex in $B$ is adjacent to $p_{23}$, $70$ other vertices in $B$, $15$ vertices in $P$ and $42$ vertices in $A$. 

We next apply an operation {\em switch} on the neighboring vertices of $p_{23}$, which erases any edge between a neighbor and non-neighbor of $p_{23}$ and creates an edge  between a neighbor and non-neighbor of $p_{23}$ when there is none. So, all edges stemming from $B$ to $A$ and $P$ are erased and all non-edges from $B$ to $A$ and $P$ become edges. Thus, $p_{23}$ becomes isolated. As such we are going to eliminate the last column and obtain codewords of length $22$, which we continue to refer as points and blocks in $P$, $A$, and $B$ having $22$, $77$, and $176$ elements respectively. We summarize the adjacency rules:
\begin{itemize}
\item[1.] A point $p\in P$ is adjacent to a block $a\in A$ if they don't share a $1$ on the same position ($56$ such occurrences), and is non-adjacent when they do; points in $P$ are disjoint; a point $p\in P$ is adjacent to a block $b\in B$, if they share a $1$ on the same position ($56$ such occurrences), and non-adjacent when they do not;  
\item[2.] A block $a\in A$ is adjacent to another block in $a^\prime \in A$ if there are no $1$'s on the same positions  ($16$ such occurrences), and non-adjacent when the two blocks share $1$'s in exactly $2$ positions; A block $a\in A$ is adjacent to a point $p \in P$ if no $1$'s are shared in the same position  ($16$ such occurrences); A block $a\in A$ is adjacent to a block $b\in B$ if exactly three $1$'s have shared positions and non-adjacent when only one $1$ is being shared in position  ($80$ such occurrences);
\item[3.] A block $b\in B$ is adjacent to another block in $b^\prime \in B$ if there is exactly one $1$ in a shared position  ($70$ such occurrences), and non-adjacent when the two blocks share $1$'s in exactly $3$ positions; A block $b\in B$ is adjacent to a point $p \in P$ if one $1$ shares the same position ($7$ such occurrences); A block $b\in B$ is adjacent to a block $a\in A$ if exactly three $1$'s have shared positions and non-adjacent when only one $1$ is being shared in position  ($35$ such occurrences).
\end{itemize}

Observe that the total number of edges for each vertex is $112$.

We next construct the code $C_{275}$ utilizing the MacLaughlin Graph structure as described above. Let $x,y,z,u,a,b$ be real variables in $[-1,1]$. We shall convert $22$-length codewords to points on $\mathbb{S}^{21}$ as follows. Vertices in $P$ will convert to points with one coordinate $z$ for $1$ and $21$ coordinates $u$ for $0$, vertices in $A$ to points with six $x$'s for $1$'s and $16$ $y$'s for $0$, and vertices in $B$ to points with seven $a$'s for $1$ and $15$ $b$'s for $0$. We shall require the inner product between two corresponding points in the code $C_{275}$ to be $-1/4$ when the corresponding vertices are adjacent and $1/6$ otherwise. 

The inner products conditions for adjacent and non-adjacent points corresponding to vertices in $A$ yields a system similar to \eqref{xy_system}
\[ \begin{split}
12 xy+10y^2&= - \frac{1}{4}\\
2x^2 + 8xy+ 12 y^2&=  \frac{1}{6},
\end{split} \]
that has four solutions, we select the one that works for the other inner products  
\begin{equation}\label{xy_values2}
x=\frac{2\sqrt{30}}{33} -\frac{\sqrt{2}}{22} \approx 0.26767, \quad \quad  y=-\frac{\sqrt{30}}{44} -\frac{\sqrt{2}}{22} \approx -0.18876.
\end{equation}

For $z$ and $u$ we get similar to \eqref{zu_system} system
\[ \begin{split}
2zu+20u^2&= \frac{1}{6}\\
z^2 + 21u^2&= 1,
\end{split} \]
with a solution
\begin{equation}\label{zu_values2}
z= \frac{3\sqrt{2}}{44} + \frac{7\sqrt{30}}{44}\approx 0.96780, \quad \quad u= \frac{3\sqrt{2}}{44} - \frac{\sqrt{30}}{132} \approx 0.05492 .
\end{equation}
Finally, inner products conditions for adjacent and non-adjacent points corresponding to vertices in $B$ yields 
\[ \begin{split}
a^2 + 12ab + 9b^2 &= - \frac{1}{4}\\
3a^2 + 8ab + 11b^2&=  \frac{1}{6},
\end{split} \]
with solution
\begin{equation}\label{ab_values2}a=- \frac{5\sqrt{30}}{88} + \frac{\sqrt{2}}{88}\approx -0.29513, \quad \quad b= \frac{7\sqrt{30}}{264} +\frac{\sqrt{2}}{88} \approx 0.16130 .\end{equation}
 
Utilizing the adjacency rules and the values from \eqref{xy_values2}, \eqref{zu_values2}, \eqref{ab_values2} we verify directly that the inner products of a point $p\in P$ with points in $A$ and $B$, are respectively
\[zy+ 6ux + 15uy =-\frac{1}{4}, \ \ zx+ 5ux + 16uy=\frac{1}{6},\] 
\[ za+ 6ua + 15ub=-\frac{1}{4},\ \  zb+7ua + 14ub=\frac{1}{6}.\]
Similarly, we verify the inner products between points in $ A$ and $B$:
\[ 3xa + 4ya + 3xb + 12yb= -\frac{1}{4},\quad xa + 5xb+  6ya + 10yb=\frac{1}{6},\]
and the unit norm conditions
\[6x^2+16y^2 =1,\quad \quad 7a^2+15b^2=1.\]

The frequency of the corresponding inner products from a fixed point in $C_{275}$ of $162$ inner products of $1/6$ and $112$ inner products of $-1/4$ follows from the property of the MacLaughlin graph and the itemized adjacency rules above.

We have also directly verified that for any fixed point $w \in C_{275}$ the inner products of the point $-w$ to the points of the code satisfy the conditions of Theorem \ref{PULB}, so the code $(22,275,4)$ attains the bound \eqref{PolarizationULB}.

\subsection{The sharp codes $(21,112,3)$ and $(21,162,3)$ -- first and second subconstituent of the  MacLaughlin graph}  \label{162} The next two codes are derived from $(22,275,4)$ in the context discussed by Delsarte, Goethals, and Seidel in \cite[Section 8]{DGS} (see also \cite{CGS}). 
Let us fix a point, say $p_{22}$ in the subsection above, and use orthogonal transformation that sends $p_{22}$ in 
$p_{275}:=(0,0,\dots,0,1)\in \mathbb{S}^{21}$, then there are two derived codes, when $\varepsilon=1/6$ and $\varepsilon=-1/4$, denoted with $A$ and $B$, respectively, with $|A|=162$ and $|B|=112$. The first configuration $A$ leads to a code $C_{162}:=(21,162,3)$, a sharp code with $162$ points, inner products $-2/7$ and $ 1/7$, that is a $3$-design. It is related to the second subconstituent of the MacLaughlin Graph, a strongly regular graph srg$(162,56,10,24)$ with vertices the points in the code and adjacency rule when the inner product is $-2/7$.

The second configuration $B$ leads to a code $C_{112}:=(21,112,3)$, one of the infinite family of sharp codes (the last row of Table 2), coming from
combinatorial configurations called generalized quadrangles (see, e.g. \cite{CGS} and references therein). Recall that 
this is a family of spherical $3$-designs in 
dimensions $n=q(q^3+1)/(q+1)$, with cardinalities $N=(q+1)(q^3+1)$ and inner products $-1/q$ and $1/q^2$, where $q$ is a 
power of a prime number. The sharp code $C_{112}$ is obtained when $q=3$, with $n=21$, $N=112$ and the inner products 
are $-1/3$ and $1/9$. It is related to the first subconstituent of the MacLaughlin Graph, a strongly regular graph srg$(112,30,2,10)$ with vertices the points in the code and adjacency rule when the inner product is $-1/3$.

We first focus on $C_{162}$. The points of $A$ are contained in an affine hyperplane orthogonal to $p_{275}$ at a distance $d=1/6$ to the origin and a radius of the circumscribed hypersphere determined by $A$ being $R=\sqrt{35}/6$. Let $e$ denote the center of mass of $A$. 

We fix any point $b\in B$. As the complementary to the MacLaughlin graph is strongly regular \linebreak srg$(275,162,105,81)$, the non-adjacent vertices $p_{275}$ and $b$ have exactly $81$ common neighbors in $A$, say $a_1,\dots ,a_{81}$. Let $g$ be the center of mass of these $81$ points. Denote the rest of the points in $A$ with $a_{82},\dots,a_{162}$. As we have already constructed $C_{275}$, direct verification shows that
\[ \frac{(g-e)\cdot (a_i -e)}{\|g-e\| R}=\frac{1}{\sqrt{21}}, \ i=1,\dots,81,\quad \frac{(g-e)\cdot (a_j -e)}{\|g-e\| R}=-\frac{1}{\sqrt{21}}, \  j=82,\dots,162.\]
Therefore, the point $\widetilde{x}:=(g-e)/\|g-e\|$, as well as its antipodal $-\widetilde{x}$ defines attaining the PULB \eqref{PolarizationULB} 
by the sharp code $C_{162}$. For each $b$ we have different split and since the second subconstituent of the MacLaughlin graph has 112 such splits, these are all of the minima.

The optimality of $C_{112}$ with respect to \eqref{PolarizationULB} is obtained similarly, this time using that two non-adjacent 
vertices in sqr$(275,112,30,56)$ have $56$ common neighbors. \hfill $\Box$

\subsection{The sharp code $(23,552,5)$  -- equiangular lines} Like $(7,56,5)$ the sharp code $C_{552}:=(23,552,5)$ is obtained by 
taking $276$ equiangular lines in $\mathbb{R}^{23}$ intersecting the unit sphere $\mathbb{S}^{22}$. 
Along with the icosahedron, these are the only configurations known to attain the absolute bound (see \cite{DGS}, \cite[Chapter 11]{GR}). 
Since the $(22,275,4)$ (MacLaughlin) code is derived from $C_{552}$, we can recover $C_{552}$ from the points of 
$C_{275}=\{x_1,x_2,\dots,x_{275}\}$ as follows. Let 
$\xi_0:=(0,\dots,0,1)\in \mathbb{S}^{22}$ be one of the points of $C_{552}$. Then the closest $275$ points will be $\xi_i:=(x_i2\sqrt{6}/5,1/5)$, $i=1,\dots,275$, where the first $22$ are coming from $P$, the next $77$ from $A$ and the last $176$ from $B$ in subsection \ref{275}. The other $276$ points are antipodal to the already present points, namely $\xi_{i+275}=-\xi_i$, $i=1,\dots,275$ and $\xi_{551}=(0,\dots,0,-1)=-\xi_0$.

Select $\widetilde{x}:=(1/5,\dots,1/5,-\sqrt{3}/5)\in \mathbb{S}^{22}$. Using \eqref{zu_values2} we compute
\[\widetilde{x}\cdot \xi_i =\frac{2\sqrt{6}(z+21u)}{25}-\frac{\sqrt{3}}{25}=\frac{\sqrt{3}}{5},\quad i=1,\dots,22.\]
Similarly, from \eqref{xy_values2} we get 
\[\widetilde{x}\cdot \xi_i =\frac{2\sqrt{6}(6x+15y)}{25}-\frac{\sqrt{3}}{25}=-\frac{\sqrt{3}}{5},\quad i=23,\dots,99.\]
Utilizing \eqref{ab_values2} we obtain
\[\widetilde{x}\cdot \xi_i =\frac{2\sqrt{6}(7a+15b)}{25}-\frac{\sqrt{3}}{25}=0,\quad i=100,\dots,275.\]
Along with $\widetilde{x}\cdot \xi_0=-\sqrt{3}/5$, and the antipodallity of the code we obtain that  there are 
$100$ points whose inner product with $\widetilde{x}$ is $-\sqrt{3}/5$, $352$ whose inner product is $0$ 
and $100$ with $\sqrt{3}/5$, so $C_{552}$ with $\widetilde{x}$ attains the PULB bound \eqref{PolarizationULB}.

For the next two codes, $(23,4600,7)$  and  $(22,891,5)$, we find it easier to determine $\widetilde{x}$ as they are embedded in $\mathbb{S}^{23}$.

\subsection{The sharp code $(23,4600,7)$ and the Leech lattice} \label{4600} Our next sharp code, $C_{4600}:=(23,4600,5)$, is a kissing 
configuration arising from the Leech lattice sharp code $C_L:=(24,196560,11)$. 

We first describe $C_L$ as the code of $196560$ minimal vectors in the Leech lattice normalized on $\mathbb{S}^{23}$
using the extended binary Golay code $\overline{\mathcal{C}}_{23}$ of length 24 and cardinality 4096. In $C_L$, we have
\begin{itemize}
\item[A.] $2^2 \cdot {24 \choose 2}=1104$ points of type 1: $(1/\sqrt{32})\left[ (\pm 4)^2 \ 0^{22}\right]$;
\item[B.] $24 \cdot 4096=98304$ points of type 2: $(1/\sqrt{32})\left[ (\mp 3)^1 (\pm 1)^{23} \right]$, where the upper signs follow the $1$'s in $\overline{\mathcal{C}}_{23}$; 
\item[C.] $2^7 \cdot 759=97152$ points of type 3: $(1/\sqrt{32})\left[ (\pm 2)^8 0^{16}\right]$, where for every one of the $759$ 
octads in $\overline{\mathcal{C}}_{23}$ we place $\pm 2$ with even number of negative signs.
\end{itemize}

We now fix one of the points in $C_L$, namely $a:=(1/\sqrt{32})[4,4,0,\dots,0]$. 
The $4600$ points $X:=\{ x_i\}_{i=1}^{4600} \subset C_L$, that have inner product $a\cdot x_i=1/2$ are:
\begin{itemize}
\item[1.] $88$ points of the type 1, namely $(1/\sqrt{32})[4,0,(\pm 4)^1\ 0^{21}]$ and $(1/\sqrt{32})[0,4,(\pm 4)^1\ 0^{21}]$;
\item[2.] $2^5 \cdot 77=2464$ points of type 2, namely $(1/\sqrt{32})[2,2,(\pm 2)^6 \ 0^{16}$, where in the $77$ octads having $1$ as first and second coordinates we place $2$ and the other six have $\pm 2$ with even number of negative signs;
\item[3.] $2\cdot 2^{10}=2048$ points of the type 3, namely $(1/\sqrt{32})[3,1,(\pm 1)^{22}]$ and $(1/\sqrt{32})[1,3,(\pm 1)^{22}]$ with signs chosen to follow the codewords in $\overline{\mathcal{C}}_{23}$ starting with $[1,0,\dots]$ or $[0,1,\dots]$.  
\end{itemize}
Note that the centroid of $X$ is $m_X :=(1/\sqrt{32})[2,2,,0,\dots,0]$ and it is a center of a hypersphere of radius $r=\sqrt{3}/2$ that circumscribes $X$. We can then express  
\[ C_{4600} = \left\{ y_i:=2(x_i - m_X)/\sqrt{3} \right\}_{i=1}^{4600}\subset {\mathbb S}^{23}.\]

The PULB quadrature data from Table 2 suggests a sub-configuration $Y$ in $X$ of $275$ points. Considering the structure 
of the McLaughlin graph and the $(22,275,4)$ code described in Subsection \ref{275}, we select the following $22+77+176$ 
construction of the McLaughlin graph-type sub-configuration $Y$. Let
\[P:= \left\{ \frac{1}{\sqrt{24}} [-2,2,(+4)^1,0^{21} ]\right\}, A:= \left\{ \frac{1}{\sqrt{24}}[0,0, (+2)^6 0^{16}] \right\}, B:= \left\{ \frac{1}{\sqrt{24}}[-1,1, (-1)^7 (+1)^{15}] \right\} ,\]
where $|P|=22$, $|A|=77$ and $|B|=176$ (in $B$ we select $-1$ for $1$ in the $176$ binary words of the Golay code $\mathcal{C}_{23}$ 
of length $23$ that have weight $7$ and start with $0$). 

We determine that the center of mass of $Y$ is 
\[ m_Y=\frac{1}{\sqrt{24}}\left[ -\frac{4}{5},\frac{4}{5},\frac{2}{5},\dots,\frac{2}{5}\right] \in \mathbb{R}^{24}.\]
Indeed, the first coordinate is $(22\cdot(-2)+77\cdot0+176\cdot(-1))/(275\sqrt{24})=-4/5\sqrt{24}$ and the second is found similarly. For the the rest of the coordinates we utilize the itemized rule 1 in Subsection \ref{275} to compute that they all will be equal to  $(4+21\cdot 2+56\cdot 0+120 \cdot 1+56\cdot (-1))/(275\sqrt{24})=2/5\sqrt{24}$. Note that $\|m_Y\|=1/\sqrt{5}$. Define
\[ \widetilde{x}:=\frac{1}{\sqrt{30}}[-2,2,1,\dots,1] \in \mathbb{S}^{23}.\]
We evaluate the inner product with points of type 1
\[\widetilde{x}\cdot  \frac{1}{\sqrt{24}} [-2,2,(+ 4)^1,0^{21}] =\frac{1}{\sqrt{5}}, \quad \widetilde{x}\cdot  \frac{1}{\sqrt{24}} [-2,2,(- 4)^1,0^{21}] =\frac{1}{3\sqrt{5}} ;\]
with type 2
\[\widetilde{x}\cdot  \frac{1}{\sqrt{24}} [0,0,(\pm 2)^6,0^{16}] =\left\{ \pm \frac{1}{\sqrt{5}}, \ \pm \frac{1}{3\sqrt{5}} \right\};\]
and type 3 (note that the $-1$'s in the last $22$ coordinates could be $7$, $11$, or $15$)
\[\widetilde{x}\cdot  \frac{1}{\sqrt{24}} [\pm 1,\mp 1,(\pm 1)^{22}] =\left\{ \pm \frac{1}{\sqrt{5}}, \ \pm \frac{1}{3\sqrt{5}} \right\}.\]
 
We directly verify the frequencies of $275$ and $2025$. Alternatively, the uniqueness of the quadrature rule implies the frequencies of the corresponding inner products. Consequently, the code $C_{4600}$ with $\widetilde{x}$ attains the PULB \eqref{PolarizationULB}. \hfill $\Box$

\subsection{The sharp code $(22,891,5)$} We describe the last sharp code in this analysis, $C_{891}:=(22,891,5)$ as derived of
$C_{4600}$. We fix a point $b:=(1/\sqrt{24})[-2,2,4,0,\dots,0]\in C_{4600}$ and determine the set $Z:=\{z_i\}_{i=1}^{891}$ of 
$891$ points from $C_{4600}$ that are closest to $b$, namely that $b\cdot z_i=1/3$. There are: 
\begin{itemize}
\item[a.] $43$ type 1 points, of which $42$ points $(1/\sqrt{24})[-2,2,0,(\pm 4)^1 0^{20}]$ and 1 point $(1/\sqrt{24})[2,-2,4,\linebreak 0,\dots, 0]$;
\item[b.] $2^4 21=336$ type 2 points $(1/\sqrt{24})[0,0,2,(\pm 2)^5 0^{16}]$ (where there are even number of $-2$'s);
\item[c.] $2^9=512$ type 3 points $(1/\sqrt{24})[-1,1,1,(\pm 1)^{21}]$. 
\end{itemize}
The center of mass of $Z$ is  
\[ m_Z=\frac{1}{3}b=\frac{1}{3\sqrt{24}}[-2,2,4,0,\dots,0] \in \mathbb{R}^{24} .\]

As $b$ is also an element of the McLaughlin graph-type code $Y$ constructed in Subsection \ref{4600}, its second subconstituent $U$ with cardinality $|U|=162$ (see Subsection \ref{162}) will be a subset of $Z$. The code $U$ will contain $21$ points $(1/\sqrt{24})[-2,2,0,(+4)^1 0^{20}]$ from $P$, $21$ points  $(1/\sqrt{24})[0,0,2,(+2)^5 0^{16}]$ from $A$, and $120$ points $(1/\sqrt{24})[-1,1,1,(-1)^7 1^{14}]$. We compute the coordinates of the centroid of $U$ to be
\[ m_U = \frac{1}{3\sqrt{24}}\left[ -3,3,3,1,\dots,1 \right] \in \mathbb{R}^{24} .\]

The calculation of the first three coordinates is obvious. For the evaluation of the other coordinates we observe that of the $21$ words in $A$ that share $2$ as a third coordinate there are exactly $5$ words that will share $2$ in any subsequent coordinate, and for any such coordinate there are $40$ words with $-1$ and $80$ words with $1$. 

We now consider the hypersphere with center $m_Z$ and radius $r=2\sqrt{2}/3$ that circumscribes $Z$ as in Subsection \ref{4600} express
\[ C_{891} =\left\{ v_i := 3(z_i-m_Z)/2\sqrt{2} \right\} \subset \mathbb{S}^{23}.\]
We select 
\[ \widetilde{x}:=\frac{m_U-m_Z}{\|m_U-m_Z\|}= \frac{1}{\sqrt{24}}\left[ -1,1,-1,1,\dots,1 \right] \in \mathbb{S}^{23}.\]
Evaluating the inner products of $\widetilde{x}$ with the points in $C_{891}$ we observe that $I(\widetilde{x},C_{891})=\{-1/3,0,1/3\}$, which coupled with the uniqueness of the quadrature rule shows that the frequencies of these inner products will be $162$, $567$, and $162$, respectively (we have also verified it directly). As a conclusion, the sharp code $C_{891}$ with the so chosen point $\widetilde{x}$ attains \eqref{PolarizationULB}.

\section{Max-min polarization of $T$-designs}\label{T-designs}
In this section we develop the framework for deriving enhanced, second-level PULB. This will allow us to find the universal minima for the remaining sharp codes listed in Table 2, but the infinite family to be considered in subsequent manuscript along with the $600$-cell.

\subsection{Spherical $T$-designs}\label{T-designs-1}

The following fundamental lemma sheds light on the relation between the moments \eqref{Moments}, 
the spherical harmonics, and the discrete potentials for Gegenbauer polynomials.

\begin{lemma}\label{BHSLemma} (e.g. \cite[Lemma 5.2.2]{BHS}, \cite{BDK}) Let $C\subset\mathbb{S}^{n-1}$ and $k$ be a positive integer. 
Let $\mathbb{H}_k^n$ be the subspace of spherical harmonics of degree $k\in \mathbb{N}$ and $Z(n,k)$ denote the dimension of this subspace. 
Then the following are equivalent:
\begin{itemize}
\item[(a)] The moment $M_k^n(C) =0$.
\item[(b)] For any orthogonal basis  $\{ Y_{kj}(x) \}_{j=1}^{Z(n,k)}$ of \, $\mathbb{H}_k^n$ 
\[\sum_{x\in C} Y_{k,j}(x)=0, \quad 1\leq j\leq Z(n,k).\]
\item[(c)] For any $x\in \mathbb{S}^{n-1}$,
\[ \sum_{y\in C} P_k^{(n)}(x\cdot y)=0.\]
\end{itemize}
\end{lemma}

The following characterization of spherical $T$-designs follows immediately from Lemma \ref{BHSLemma}. 

\begin{proposition} \label{prop23}
Let $T\subset \mathbb N$ be a nonempty index set. Then $C \subset \mathbb{S}^{n-1}$ is a spherical $T$-design if and only if  \eqref{defin_f} holds for all $f \in \mathcal{P}_{T}$.\end{proposition} 

We will use this property to derive enhanced PULBs. 

\begin{definition} \label{LowerClass} 
Let $T\subset \mathbb{N}$ be an index set and $h$ be a potential function. Denote by $\mathcal{L}(n,T,h)$ 
the class of {\em lower admissible polynomials} $f(t)$ such that 
\begin{itemize}
\item[(A1)] $f\in \mathcal{P}_T$;
\item[(A2)] $f(t) \leq h(t)$ for every $t \in [-1,1]$.
\end{itemize}
\end{definition}

Utilizing this definition  and Proposition \ref{prop23} we derive the following Delsarte-Yudin type lower bound on the polarization potential of $h(t)$.

\begin{proposition} \label{prop32}
Let $h(t)$ be a potential function, $T\subset \mathbb{N}$ be an index set, and $f \in {\mathcal L}(n,T, h)$ be a lower admissible polynomial. Then for all spherical $T$-designs $C \subset \mathbb{S}^{n-1}$ the following lower bound holds:
\[ U_h(x,C) \geq  U_f(x,C) = f_0|C|, \quad x\in \mathbb{S}^{n-1}. \]
Consequently,
\begin{equation}\label{DY_LB}
\mathcal{Q}_h(C) \geq \max_{f\in {\mathcal L}(n,T,h)} \{ f_0|C| \}. 
\end{equation}
\end{proposition}

\begin{proof}
This is immediate from Definition \ref{T-design}, Definition \ref{LowerClass} and Proposition \ref{prop23}.
\end{proof}

The following corollary is an immediate consequence of \eqref{DY_LB}.

\begin{corollary} \label{cor4-5} Let $T\subset \mathbb{N}$ be an index set and suppose that the collection of spherical $T$-designs $C \subset \mathbb{S}^{n-1}$ of cardinality $|C|=N$ is non-empty. Then
\begin{equation} \label{Polarization_LB1}\underline{\mathcal Q}_h(n,N,T) \geq N\max_{f\in {\mathcal L}(n,T,h)}   f_0.
\nonumber
\end{equation}
\end{corollary}

This gives rise to the following linear program:
\begin{align}\label{LP_ULB}
{\rm given}: & \quad n,  \ T, \ h \nonumber \\
{\rm maximize}: & \quad f_0  \\
{\rm subject \ to:} &\quad  f \in {\mathcal L}(n,T,h) \nonumber
\end{align}

\subsection{PULB-spaces}

In \cite{BDHSS_P} we were able to solve this linear program explicitly in the important case of spherical $\tau$-designs, i.e. $T=\{1,\dots,\tau\}$, 
for certain classes of potentials, leading to Theorem \ref{PULB}.  Gauss-Jacobi quadrature formulas play an essential role in the solution.

Thus, the main ingredients in our approach to the program \eqref{LP_ULB} will be the existence of an appropriate quadrature rule 
exact on $\mathcal{P}_T$ and an interpolation polynomial from $\mathcal{P}_T$ staying below the potential $h$. In the case of spherical $\tau$-designs the featured existence on $\mathcal{P}_\tau$ is guaranteed by the classical Gauss-Jacobi quadratures and the Hermite interpolation polynomials. 

In Section \ref{pos-derivative} we showed that most known sharp codes attain the bound \eqref{PolarizationULB} for potentials $h$ 
such that $h^{(\tau+1)} \geq 0$. Yet, there are notable exceptions, such as the icosahedron and the kissing configurations of $E_8$ 
and the Leech lattice that do not attain this first level PULB \eqref{PolarizationULB}. 
However, we will show that they attain an enhanced PULB for properly chosen 
index set $T\supset\{1,\dots,\tau\}$. 

Our framework for deriving an enhanced PULB on $\mathcal{P}_T\supset \mathcal{P}_\tau$ will follow similar approach as the one for potential energy ULB in \cite{BDHSSMathComp}. We start with the concept of a PULB-space.

\begin{definition}
\label{def-PULB-space}
Let $n\geq 2$ be a positive integer and $T\subset \mathbb{N}$ an index set. A space $\mathcal{P}_T$ is a {\em PULB-space} for dimension $n$ associated with the index set $T$
if the following two conditions hold:

\rm{(i)} there exists a quadrature rule with nodes in $[-1,1]$ and positive weights that is exact on $\mathcal{P}_T$.

\rm{(ii)} for any absolutely monotone function $h$ there exists some  $f\in \mathcal{L}(n,T,h)$ that agrees with $h$ at the nodes of the quadrature rule from (i).

\end{definition}

It follows from Theorem~\ref{PULB} that $\mathcal{P}_\tau$ is a PULB-space for the set $T=\{1,2,\ldots,\tau\}$.  Although the index set   can 
be infinite (e.g., for  antipodal codes), we will consider only finite $T$ in this paper.

When $\mathcal{P}_T$ is a PULB-space, then we may solve the linear program \eqref{LP_ULB}. 
\begin{theorem} \label{PULBQR}
Suppose for some $T\subset \mathbb{N}$ that $\mathcal{P}_T$ is a PULB-space for dimension $n$ with quadrature rule given by  nodes $\{\alpha_i\}_{i=1}^\ell$ and weights $\{\rho_i\}_{i=1}^\ell$.    For any potential $h$ absolutely monotone on $[-1,1]$, we have
\begin{equation} \label{PULB1}
\max_{f \in {\mathcal L}(n,T,h)}f_0 =  \sum_{i=1}^{\ell} \rho_i h(\alpha_i).
\end{equation}
\end{theorem}
\begin{proof}Let $h$ be absolutely monotone on $[-1,1]$ and  $f\in \mathcal{L}(n,T,h)$.   Then, 
$$f_0=\sum_{i=1}^{\ell} \rho_i f(\alpha_i)\le \sum_{i=1}^{\ell} \rho_i h(\alpha_i). $$
On the other hand, choosing  $f\in \mathcal{L}(n,T,h)$ that agrees  with $h$ at the nodes  $\{\alpha_i\}_{i=1}^\ell$
shows \eqref{PULB1}.
\end{proof}

\begin{theorem} \label{facetThm}  Suppose $\mathcal{P}_T$ is a PULB-space for dimension $n$ with (unique) quadrature rule given by  nodes $\{\alpha_i\}_{i=1}^\ell$ and weights $\{\rho_i\}_{i=1}^\ell$ and let $\alpha_\ell=\max_i\alpha_i$.  Furthermore, suppose $C\subset \mathbb{S}^{n-1}$ is a spherical $T$-design that attains the PULB 
\begin{equation}\label{tightPULB1}
\mathcal{Q}_h(C)=  \sum_{i=1}^{\ell} \rho_i h(\alpha_i).
\end{equation}

Then  $y\in \mathbb{S}^{n-1}$ is a minimum point of $U_h( \cdot,C)$ (i.e., $U_h( y,C)=\mathcal{Q}_h(C)$)    if and only if the collection of inner products $I(y, C)$, coincides with  $\{\alpha_i \}_{i=1}^\ell$ 
and the frequencies of these inner products are $\{N\rho_i\}_{i=1}^\ell$. 
Furthermore, if $y$ is   a minimum point of $U_h( \cdot,C)$, then:
\begin{itemize}
\item[{(i)}] $y$ is a furthest point in $\mathbb{S}^{n-1}$ from $C$; i.e., $d(y,C)=\max_{x\in\mathbb{S}^{n-1}} d(x,C)$, where $d(x,C):=\min_{z\in C} \|x-z\|$. 
\item[{(ii)}]  the subsets $C_{y,\alpha_i}:=\{z\in C: z\cdot y=\alpha_i\}$ have centroids $\frac{1}{N\rho_i}\sum_{z\in C_{y,\alpha_i}}z=\alpha_i y$, $i=1,\dots,\ell$. 
\item[{(iii)}]  
The convex hull of $C_{y,\alpha_\ell}$ forms a facet of the convex hull of $C$ (and hence $N\rho_\ell \ge n$). 
\end{itemize}
\end{theorem}

\begin{proof} 
(i) Suppose to the contrary that $d(y,C)<d(x,C)$ for some ${x\in\mathbb{S}^{n-1}} $. For the (absolutely monotone) Riesz potentials $h_m (t):=(2-2t)^{-m/2}$ with parameter $m>0$ we have
\[ \frac{N}{d(x,C)^m}\geq \sum_{z\in C}h_m (x\cdot z) \geq N\sum_{i=1}^{\ell} \rho_i h_m(\alpha_i) \geq \frac{N\rho_\ell}{d(y,C)^m}.\]
%
Taking an $m$-th root and letting $m\to \infty$ we derive a contradiction.

(ii) We first prove the statement when $i=\ell$. As a global minimum, $y$ is also a local minimum of $U_{h_m}(x,C)$ on $\mathbb{S}^{n-1}$, so the gradient is a constant multiple of $y$, i.e.
\[ \nabla U_{h_m} (y,C)=\sum_{x\in C}\frac{2y-2x}{\|  y-x\|^{m+2}} = C_m y.\]
To find the constant $C_m$ we take a dot product with $y$ and use that $\|x-y\|^2=2-2x\cdot y$ for any two vectors on the unit sphere to conclude that 
\[C_m=\sum_{x\in C} \frac{1}{\|y-x\|^m}.\]
Thus, we derive that 
\begin{equation}\label{BalancedEq} \sum_{x\in C_{y,\alpha_\ell}}\frac{2y-2x}{d(y,C)^{2}} +\sum_{x\in C\setminus C_{y,\alpha_\ell} }\frac{(2y-2x)d(y,C)^m}{\|  y-x\|^{m+2}}=(N\rho_\ell ) y+ \sum_{x\in C\setminus C_{y,\alpha_\ell}} \frac{d(y,C)^m}{\|y-x\|^m} y .\end{equation}
As we let $m\to \infty$, the second terms on both sides approach zero and we arrive at the formula
\[ \frac{1}{N\rho_\ell}\sum_{x\in C_{y,\alpha_\ell}} x = \left( 1-\frac{d(y,C)^2}{2} \right) y=\alpha_\ell y.\]
This shows the statement for $i=\ell$. Applying this equation to \eqref{BalancedEq} we get 
\[ \sum_{x\in C\setminus C_{y,\alpha_\ell} }\frac{(2y-2x)d(y,C)^m}{\|  y-x\|^{m+2}}=\sum_{x\in C\setminus C_{y,\alpha_\ell}} \frac{d(y,C)^m}{\|y-x\|^m} y \]
and after dividing both sides by $d(y,C)^m$ we can apply similar arguments to conclude the validity for all $i=1,\dots,\ell$.

(iii) This assertion follows from \cite[Lemma 1]{BS21}.
\end{proof}

We shall use the Riesz potentials $h_m$ introduced in the proof above in the following corollary.

\begin{corollary} If $\mathcal{P}_T$ is a PULB-space for  dimension $n$ then the corresponding quadrature rule from Definition \ref{def-PULB-space} is unique.
\end{corollary}

\begin{proof} Suppose there are two quadrature rules with nodes $\{\alpha_i \}_{i=1}^\ell, \{\alpha_i^\prime \}_{i=1}^{\ell^\prime}$ (labeled in increasing order) and weights $\{\rho_i \}_{i=1}^\ell$ and $ \{\rho_i^\prime \}_{i=1}^{\ell^\prime}$, respectively. Then \eqref{PULB1} implies that for any absolutely monotone $h$ we have
\begin{equation}\label{UniquenessQR1} \sum_{i=1}^{\ell} \rho_i h(\alpha_i)= \sum_{i=1}^{\ell^\prime} \rho_i^\prime h(\alpha_i^\prime). \end{equation}
We first prove that $\alpha_\ell=\alpha_\ell^\prime$ and $\rho_\ell=\rho_\ell^\prime$. If $\alpha_\ell\not =\alpha_\ell^\prime$, without loss of generality we may assume $\alpha_\ell<\alpha_\ell^\prime$. Using $h(t)=h_m(t)$ in \eqref{UniquenessQR1} we get 
\begin{equation}\label{UniquenessQR2} \sum_{i=1}^{\ell} \rho_i (2-2\alpha_i)^{-m/2}= \sum_{i=1}^{\ell^\prime} \rho_i^\prime (2-2\alpha_i^\prime)^{-m/2}. \end{equation}
Multiplying by $(2-2\alpha_\ell^\prime)^{m/2}$ the equation \eqref{UniquenessQR2} and letting $m\to +\infty$ we derive $\rho_\ell^\prime=0$, which is a contradiction with the assumption that the weights are positive. If $\alpha_\ell=\alpha_\ell^\prime$, the same process will yield that $\rho_\ell=\rho_\ell^\prime$. Once we establish that $\alpha_\ell=\alpha_\ell^\prime$ and $\rho_\ell=\rho_\ell^\prime$, we can cancel the corresponding terms in \eqref{UniquenessQR2} and proceed analogously to deduce the uniqueness.
\end{proof}

\subsection{Quadrature rule on $\mathcal{P}_T$} 
We now extend the Gauss-Jacobi quadrature on $\mathcal{P}_\tau$ to a larger subspace of polynomials.  Motivated by the cases  of sharp codes for which the PULB is not attained, we shall consider $\tau$-designs for {\em odd}  $\tau=2k-1$ that have additional zero moments of higher order. Recall that for any integer $k$ the {\em Gauss-Jacobi quadrature rule} 
\[ \int_{-1}^1 f(t)\, d\mu_n(t)=\sum_{i=1}^k \rho_i f(\alpha_i) \]
holds true for all $f\in \mathcal{P}_{2k-1}$. Here the nodes $\{ \alpha_i\}, i=1,\dots, k$ are the roots of the Gegenbauer polynomial $P_k^{(n)}(t)$ and $\rho_i$ are found as in \eqref{RhoWeights}.

Specifically, we consider  {\em skip one-add two} index set of the form
\[T^k:=\{ 1,\dots, 2k-1, 2k+1, 2k+2\}=\{ 1,2,\dots,  2k+2\}\setminus\{2k\}.\]
The polynomial subspace is given as
\begin{equation}\label{Skip1Add2}
\mathcal{P}_{T^k} := \mathcal{P}_{2k-1}
\oplus  \mbox{span} \left(P_{2k+1}^{(n)}, P_{2k+2}^{(n)}\right).
\end{equation}

We shall focus on necessary conditions for the existence of a quadrature rule exact on  $\mathcal{P}_{T^k}$. In the process, we will establish existence and uniqueness. 

\begin{lemma} \label{lem4-7}
Suppose $n\geq 3$. Any quadrature rule that is exact on $\mathcal{P}_{T^k}$ has at least $k+1$ distinct nodes.
\end{lemma}
\begin{proof} Let $\{ \beta_i,\theta_i\}_{i=1}^\ell$ be such a quadrature rule. If $\ell < k$, 
then $f(t):=(t-\beta_1 )^2 \dots (t-\beta_\ell )^2 \in \mathcal{P}_{T^k}$ and hence
\[
\int_{-1}^1 (t-\beta_1 )^2 \dots (t-\beta_\ell )^2\, d\mu_n(t) =\sum_{i=1}^\ell \theta_i f(\beta_i) = 0,
\]
which is impossible. 

If $\ell=k$, then since $\mathcal{P}_\tau\subset \mathcal{P}_{T^k}$, the quadrature rule is the Gauss-Jacobi (or Gegenbauer) quadrature with nodes the zeros $\alpha_i$, $i=1,\dots,k$, of $P_k^{(n)}(t)$, and respective weights $\rho_i$, $i=1,\dots,k$. Let us introduce the two linear functionals acting on polynomials
\[ I(f):=\int_{-1}^1 f(t)\, d\mu_n(t),\quad \quad \quad \quad QR(f):=\sum_{i=1}^k \rho_i f(\alpha_i). \]
Clearly, $I(P_i^{(n)})=QR(P_i^{(n)})=0$, $i\in T$. We also have that $I(P_{2k}^{(n)})=0$, which implies that $QR(P_{2k}^{(n)})\not=0$, 
for otherwise the quadrature rule will be exact on $\mathcal{P}_{2k+1}$ and the uniqueness of the Gauss-Jacobi quadrature will lead to a contradiction. Consider the polynomial 
\begin{equation}\label{ProductExpansion} 
f(t):=P_k^{(n)}(t)P_{k+2}^{(n)}(t)=f_{2k+2}P_{2k+2}^{(n)}(t)+f_{2k}P_{2k}^{(n)}(t)+\dots +f_2P_{2}^{(n)}(t).
\end{equation}
Observe that $f_0=I(f)=0$. From \cite{G} all the coefficients in the Gegenbauer expansion are nonnegative and $f_{2k+2}$ and $ f_2$ are strictly positive. 

We claim that $f_{2k}> 0$ as well. Suppose there is $k,n$ such that $f_{2k}=0$ in \eqref{ProductExpansion}. For brevity, let $\alpha:=n/2-1$. We shall utilize the first two terms of the expansion (see \cite[Equation (2.3)]{H})
\[ P_k^{(n)}(t)=\frac{(\alpha)_k}{(2\alpha)_k}(2t)^k-\frac{k(k-1)(\alpha)_{k-1}}{(2\alpha)_k}(2t)^{k-2}+\cdots \ ,\]
where $(\alpha)_k:=\alpha(\alpha+1)\cdot \dots \cdot (\alpha+k-1)$ is the Pochhammer symbol, to conclude that
\begin{equation} \label{GegExpansion}
f(t)=\frac{(\alpha)_k(\alpha)_{k+2}}{(2\alpha)_k(2\alpha)_{k+2}}\left[ (2t)^{2k+2}-  \left( \frac{(k+2)(k+1)}{\alpha+k+1} +\frac{ k(k-1)}{\alpha+k-1} \right) (2t)^{2k}\right] +\cdots \ .
\end{equation}
Similarly, \eqref{ProductExpansion} with $f_{2k}=0$ yields
\[ f(t)=f_{2k+2}\frac{(\alpha)_{2k+2}}{(2\alpha)_{2k+2}}\left[ (2t)^{2k+2}-  \frac{(2k+2)(2k+1)}{\alpha+2k+1}(2t)^{2k}\right] +\cdots \ .\]
Thus, we obtain the equation 
\[\frac{(2k+2)(2k+1)}{\alpha+2k+1}=\frac{(k+2)(k+1)}{\alpha+k+1} +\frac{ k(k-1)}{\alpha+k-1}.\]
It is easy to verify that the latter equation has two solutions, $\alpha=0$ and $\alpha=-k$. This contradiction proves the claim that $f_{2k}>0$.

To complete the proof, observe that $QR(f)=0$ because $P_k^{(n)}$ is annihilating polynomial for the quadrature rule. 
From \eqref{ProductExpansion} however, $QR(f)=f_{2k}QR(P_{2k}^{(n)})\not= 0$. The derived contradiction affirms the lemma.
\end{proof}

\subsection{Skip 1 -- Add 2 quadrature framework} We next exhibit a framework on how to obtain the quadrature nodes should we have a skip one-add two subspace depicted in \eqref{Skip1Add2}. We shall also establish that this subspace is a tight PULB-space. 

We continue with the necessary and sufficient conditions for existence of a quadrature in Definition \ref{def-PULB-space} (i). Suppose we have exactly $k+1$ quadrature nodes $\beta_1<\dots< \beta_{k+1} $ in $(-1,1)$
and weights $\gamma_i$, $i=1,2,\dots,k+1$ (the existence is established in Lemma \ref{PositiveWeights}), such that 
\begin{equation}\label{S1A2Q} \int_{-1}^1 f(t)d\mu_n (t)=\sum_{i=1}^{k+1}\gamma_i f(\beta_i) \end{equation}
for any $f\in \mathcal{P}_{T^k}$. Define the annihilating polynomial associated with the quadrature as
\[ q_{k+1}(t):=C_k (t-\beta_1)\dots(t-\beta_{k+1})=P_{k+1}^{(n)}(t)+aP_k^{(n)}(t)+bP_{k-1}^{(n)}(t)+u_{k-2}(t),\]
where $C_k$ is a normalizing constant and $u_{k-2}(t)$ is a polynomial of degree at most $k-2$. We shall show that $u_{k-2}(t)\equiv 0$. Indeed, for any $i=0,1,\dots,k-2$ the polynomial $q_{k+1}P_i^{(n)}\in \mathcal{P}_{T^k}$, so the quadrature \eqref{S1A2Q} holds true, implying that 
\[ 0=\int_{-1}^1 q_{k+1}(t) P_i^{(n)}(t)\, d\mu_n(t)=\int_{-1}^1 u_{k-2}(t) P_i^{(n)}(t)\, d\mu_n(t),\quad i=0,1,\dots,k-2.\]
Therefore, $u_{k-2}(t)\equiv 0$ and 
\begin{equation}\label{q_k+1} q_{k+1}(t)=P_{k+1}^{(n)}(t)+aP_k^{(n)}(t)+bP_{k-1}^{(n)}(t)\end{equation}
for some $a,b\in \mathbb{R}$.

To determine $a$ and $b$, choose any $f\in \mathcal{P}_{T^k}$. Then 
\[ f(t)=q_{k+1}(t)\left(d_{k+1}P_{k+1}^{(n)}(t)+d_k P_k^{(n)}(t)+d_{k-1}P_{k-1}^{(n)}(t)\right)+v_{2k-1}(t),\]
where $v_{2k-1}(t)$ is a polynomial of degree at most $2k-1$. Clearly, $f-v_{2k-1}\in \mathcal{P}_{T^k}$, so 
\begin{equation}\label{Qe} \int_{-1}^1 q_{k+1}(t)\left(d_{k+1}P_{k+1}^{(n)}(t)+d_k P_k^{(n)}(t)+d_{k-1}P_{k-1}^{(n)}(t)\right)\, d\mu_n (t)=0\end{equation}
from the quadrature \eqref{S1A2Q}. As $f$ is orthogonal to $P_{2k}^{(n)}$ we also have
\begin{equation}\label{Oe} \int_{-1}^1 q_{k+1}(t)\left(d_{k+1}P_{k+1}^{(n)}(t)+d_k P_k^{(n)}(t)+d_{k-1}P_{k-1}^{(n)}(t)\right)P_{2k}^{(n)}(t)\, d\mu_n (t)=0.\end{equation}

Next, since $f$ was arbitrary in $\mathcal{P}_{T^k}$ we can select two (basic) triples $(d_{k+1},d_k,d_{k-1})=(1,0,c)$ and $(d_{k+1},d_k,d_{k-1})=(0,1,d)$ 
so that equations \eqref{Qe} and \eqref{Oe} hold true. This yields a system of equations
\[
\begin{split}
\int_{-1}^1 \left(P_{k+1}^{(n)}(t)+aP_k^{(n)}(t)+bP_{k-1}^{(n)}(t)\right) \left(P_{k+1}^{(n)}(t)+cP_{k-1}^{(n)}(t)\right)\, d\mu_n (t)&=0\\
\int_{-1}^1 \left(P_{k+1}^{(n)}(t)+aP_k^{(n)}(t)+bP_{k-1}^{(n)}(t)\right) \left(P_{k}^{(n)}(t)+dP_{k-1}^{(n)}(t)\right)\, d\mu_n (t)&=0\\
\int_{-1}^1 \left(P_{k+1}^{(n)}(t)+aP_k^{(n)}(t)+bP_{k-1}^{(n)}(t)\right) \left(P_{k+1}^{(n)}(t)+cP_{k-1}^{(n)}(t)\right)P_{2k}^{(n)}(t)\, d\mu_n (t)&=0\\
\int_{-1}^1 \left(P_{k+1}^{(n)}(t)+aP_k^{(n)}(t)+bP_{k-1}^{(n)}(t)\right) \left(P_{k}^{(n)}(t)+dP_{k-1}^{(n)}(t)\right)P_{2k}^{(n)}(t)\, d\mu_n (t)&=0,
\end{split}
\]
which reduce to
\begin{equation}\label{System}
\begin{split}
\|P_{k+1}^{(n)}\|^2+bc\|P_{k-1}^{(n)}\|^2 &=0\\
a\|P_{k}^{(n)}\|^2+bd\|P_{k-1}^{(n)}\|^2 &=0\\
\langle (P_{k+1}^{(n)})^2,P_{2k}^{(n)}\rangle+(b+c)\langle P_{k+1}^{(n)}P_{k-1}^{(n)},P_{2k}^{(n)}\rangle&=0\\
a\langle (P_{k}^{(n)})^2,P_{2k}^{(n)}\rangle+d\langle P_{k+1}^{(n)}P_{k-1}^{(n)},P_{2k}^{(n)}\rangle&=0.\\
\end{split}
\end{equation}
Let $a_k$ denote the leading coefficient of the Gegenbauer polynomial of degree $k$, i.e. $P_{k}^{(n)}(t)=a_kt^k+\cdots$. 
The last equation of \eqref{System} then becomes 
\[ (aa_k^2+da_{k+1}a_{k-1})\langle t^{2k},P_{2k}^{(n)}\rangle=0.\]
Since $\langle t^{2k},P_{2k}^{(n)}\rangle \not=0$, we have that $d=-aa_k^2/(a_{k+1}a_{k-1})$. Utilizing the following formulas (see \cite{Lev} and \cite[Formula (5.2.3)]{BHS})
\[ \frac{a_{k+1}}{a_k}=\frac{n+2k-2}{n+k-2},\quad \|P_{k}^{(n)}\|^2=\frac{n+k-2}{n+2k-2}{k+n-2 \choose k}^{-1},\]
in the second and fourth equations of \eqref{System} we obtain two cases: (i) $a=d=0$; or (ii) $b=k/(k+n-2)$. The second case leads to a contradiction with the first and the third equations of \eqref{System} (see also \eqref{k/(k+n-2)} below). Therefore, we conclude that $a=d=0$, which reduces the system to 

\[ \begin{split}
\|P_{k+1}^{(n)}\|^2+bc\|P_{k-1}^{(n)}\|^2 &=0\\
\langle (P_{k+1}^{(n)})^2,P_{2k}^{(n)}\rangle+(b+c)\langle P_{k+1}^{(n)}P_{k-1}^{(n)},P_{2k}^{(n)}\rangle&=0,
\end{split} \]
which can be simplified utilizing \eqref{GegExpansion} to (recall that $ \alpha =n/2-1$)
\begin{equation} \label{bcSystem}
\begin{split}
bc&=-\frac{(k+1)k(\alpha+k-1)}{(2\alpha+k)(2\alpha+k-1)(\alpha+k+1)}\\
b+c&=-\frac{2\alpha(k+1)^2(\alpha+k-1)}{(2\alpha+k)(2\alpha+k-1)(\alpha+2k+1)}.
\end{split} 
\end{equation}

The lemma below shows that when we select $b$ as the positive of the two solutions of \eqref{bcSystem}, the polynomial $q_{k+1}(t)$ has $k+1$ distinct roots in $(-1,1)$. 

\begin{lemma}\label{bcRoots} 
Let $b>0>c$ be the two solutions of \eqref{bcSystem}. 
 Then 
\begin{equation}\label{NextLevelPoly}
q_{k+1}(t):=P_{k+1}^{(n)}(t)+bP_{k-1}^{(n)}(t) =\frac{(\alpha)_{k+1}2^{k+1}}{(2\alpha)_{k+1}}(t-\beta_1)\dots(t-\beta_{k+1}),
\end{equation}
where $\beta_i $ are distinct, symmetric about the origin, and in the interval $(-1,1)$. 
\end{lemma}

\begin{proof}
The symmetry of the roots is an immediate consequence of the explicit form of $q_{k+1}$. It is easily shown that the system \eqref{bcSystem} has two real solutions. Without loss of generality assume $b>0>c$. Using the three-term recurrence formula for the Gegenbauer polynomials (with the normalization $P_i^{(n)}(1)=1$)
\begin{equation} \label{3term}
(2\alpha+k)P_{k+1}^{(n)}(t)=2(\alpha+k)tP_k^{(n)}(t)-kP_{k-1}^{(n)}(t),
\end{equation}
we shall write
\[q_{k+1}(t)= P_{k+1}^{(n)}(t)+bP_{k-1}^{(n)}(t)= \frac{2(\alpha+k)}{2\alpha+k}tP_k^{(n)}(t)-\left( \frac{k}{2\alpha+k} -b\right) P_{k-1}^{(n)}(t).\]

We claim that the constant $c_1:=k/(2\alpha+k)-b$ is positive. Indeed, consider the quadratic polynomial $g_2(t):=(t-b)(t-c)$. Using  \eqref{bcSystem}, we compute directly that
\begin{equation} \label{k/(k+n-2)}g_2\left( \frac{k}{2\alpha+k} \right)=\frac{2k^2 \alpha(k + 2)(\alpha + k)^2}{(2\alpha + k)^2(2\alpha + k - 1)(\alpha + 2k + 1)(\alpha + k + 1)}>0.\end{equation}
As $g_2(t)\leq 0 $ on $[c,b]$, $c<0<b$, and $k/(2\alpha+k)>0$, we derive the claim.

We will obtain now the existence of the $k+1$ roots of $q_{k+1}$ by careful accounting of the interlacing property of the zeros $\{\alpha_i\}$ of $P_k^{(n)}$ and $\{\eta_j\}$ of $P_{k-1}^{(n)}$. We know that $-1<\alpha_1<\eta_1<\alpha_2<\dots<\eta_{k-1}<\alpha_k$. Clearly, as $q_{k+1}(1)=1+b>0$ and $q_{k+1}(\alpha_k)=-c_1P_{k-1}^{(n)}(\alpha_k)<0$, we will have one zero in the interval $(\alpha_k,1)$. From $q_{k+1}(\alpha_{k-1})=-c_1P_{k-1}^{(n)}(\alpha_{k-1})>0$ we account for another root of $q_{k+1}$ in the interval $(\alpha_{k-1},\alpha_k)$. We complete the proof of the lemma analogously utilizing the symmetry of the Gegenbauer polynomials.
\end{proof}

\begin{lemma}\label{PositiveWeights} The zeros of $q_{k+1}(t)$ defined in \eqref{NextLevelPoly} induce a quadrature rule on $\mathcal{P}_{T^k}$ with positive weights.
\end{lemma}

\begin{proof} In Lemma \ref{bcRoots} we established that the polynomial $q_{k+1}$ defined in \eqref{q_k+1} has $k+1$ distinct roots $\{\beta_i\}\subset (-1,1)$, which we use as nodes of the quadrature \eqref{S1A2Q}. Since the Lagrange basic polynomials $L_i (t)$, $i=1,\dots,k+1$, are of degree $k$, they are contained in $\mathcal{P}_{T^k}$, so the weights $\gamma_i:=\int_{-1}^1 L_i(t)\, d \mu_n (t)$, $i=1,\dots,k+1$ are uniquely determined. Moreover, they are symmetric because of the symmetry of the nodes. 
What is left is to prove for the existence of the quadrature in Definition \ref{def-PULB-space} (i) is the positivity of the weights $\gamma_i$. 

Let us consider first the case when $k$ is odd. Then the nodes and the weights are evenly paired as $\beta_i=-\beta_{k+2-i}$ and $\gamma_i=\gamma_{k+2-i}$, $i=1,\dots,(k+1)/2$. The polynomial 
\[u_i (t):=\frac{(t-\beta_1)^2\dots(t-\beta_{k+1})^2}{(t-\beta_i)^2(t-\beta_{k+2-i})^2} = \frac{(t-\beta_1)^2\dots(t-\beta_{k+1})^2}{(t^2-\beta_i^2)^2}\] is of degree $2k-2$ and belongs to $\mathcal{P}_{T^k}$. Applying the quadrature rule we obtain
\[ 0<\int_{-1}^1 u_i(t)\, d\mu_n(t)=\gamma_i u_i(\beta_i)+\gamma_{k+2-i}u_i(\beta_{k+2-i})=2\gamma_i u_i(\beta_i),\]
which implies $\gamma_i>0$.

When $k:=2m$ is even, the same argument yields the positivity of all weights but $\gamma_{m+1}$ which is
corresponding to the node $\beta_{m+1}=0$. Recall that the zeros $\{\alpha_i\}_{i=1}^k$ of $P_k^{(n)}$ define the 
Gauss-Gegenbauer quadrature with positive weights $\{\rho_i \}_{i=1}^k$
\[ \int_{-1}^1 f(t)\, d\mu_n(t)=\sum_{i=1}^k \rho_if(\alpha_i)\]
that is exact for all polynomials $f$ of degree at most $2k-1$. This is the quadrature utilized to derive the first level PULB. The proof of Lemma
 \ref{bcRoots}  established that $\{ \alpha_i \}$ and $\{ \beta_i\}$ interlace, i.e. 
\[\beta_1<\alpha_1<\beta_2<\dots<\alpha_m<0=\beta_{m+1}<\alpha_{m+1}<\dots<\beta_{2m}<\alpha_{2m}<\beta_{2m+1}.\] 
We also have that the nodes $\{\alpha_i\}$ and the weights $\{ \rho_i \}$ are symmetric with respect to the origin, namely $\alpha_i=-\alpha_{k+1-i}$ and $\rho_i=\rho_{k+1-i}$, $i=1,\dots,m$.  We now select the degree $2k-2$ polynomial
\[ f(t):=\frac{(t-\beta_1)\dots(t-\beta_m)(t-\beta_{m+2})\dots (t-\beta_{2m+1})(t-\alpha_2)\dots(t-\alpha_{2m-1})}{(-\beta_1)\dots (-\beta_m)(-\beta_{m+2})\dots(-\beta_{2m+1})(-\alpha_2)\dots (-\alpha_{2m-1})} \in \mathcal{P}_{T^k}.\]
Utilizing the symmetry we simplify 
\[  f(t):=-\frac{(t^2-\beta_{m+2}^2)\dots (t^2-\beta_{2m+1}^2)(t^2-\alpha_{m+1}^2)\dots(t^2-\alpha_{2m-1}^2)}{\beta_{m+2}^2\dots\beta_{2m+1}^2\alpha_{m+1}^2\dots \alpha_{2m-1}^2}.\]
Applying the two quadrature rules with respect to the measure $\mu_n$ to $f$ we obtain
\begin{align*}
\gamma_{m+1}&=\int_{-1}^1 f(t)\, d\mu_n(t)=2\rho_{2m}f(\alpha_{2m})\\
&=-2\rho_{2m}\frac{(\alpha_{2m}^2-\beta_{m+2}^2)\dots (\alpha_{2m}^2-\beta_{2m+1}^2)(\alpha_{2m}^2-\alpha_{m+1}^2)\dots(\alpha_{2m}^2-\alpha_{2m-1}^2)}{\beta_{m+2}^2\dots\beta_{2m+1}^2\alpha_{m+1}^2\dots \alpha_{2m-1}^2}>0,
\end{align*}
because all terms in the numerator are positive except for $\alpha_{2m}^2-\beta_{2m+1}^2<0$. This completes the proof of the lemma.
\end{proof}

The last step in the preparation of the main result in this subsection is a positive definiteness result needed for the interpolation 
property from  Definition \ref{def-PULB-space} (ii).

\begin{lemma}\label{PosDef} The polynomial $g_k(t):=(t-\beta_1)(t-\beta_2)\dots(t-\beta_k)$ has positive Gegenbauer coefficients.
\end{lemma}
\begin{proof} 
Denote the Gegenbauer expansion
\begin{equation} \label{GegPosDef}
g_k(t)=\frac{(2\alpha)_{k+1}}{2^{k+1}(\alpha)_{k+1}}\frac{q_{k+1}(t)}{t-\beta_{k+1}} = \sum_{i=0}^k c_i P_i^{(n)}(t).
\end{equation}
We first note that from \eqref{GegExpansion}
\[ q_{k+1}(t)=P_{k+1}^{(n)}(t)+bP_{k-1}^{(n)}(t) =\frac{(\alpha)_{k+1}2^{k+1}}{(2\alpha)_{k+1}}t^{k+1}+\left( \frac{(b\alpha)_{k-1}2^{k-1}}{(2\alpha)_{k-1}} -\frac{((k+1)k\alpha)_{k}2^{k-1}}{(2\alpha)_{k+1}}\right) t^{k-1}+\dots,\]
which implies that $\beta_1+\beta_2+\dots+\beta_{k+1}=0$. As
\begin{align*} g_k(t)&= t^k-(\beta_1+\dots+\beta_k)t^{k-1}+\dots=t^k+\beta_{k+1}t^{k-1}+\dots\\
&=\frac{(2\alpha)_k}{2^k(\alpha)_k}P_k^{(n)}(t)+\beta_{k+1}\frac{(2\alpha)_{k-1}}{2^{k-1}(\alpha)_{k-1}}P_{k-1}^{(n)}(t)+\dots ,\end{align*}
we have that $c_k=(2\alpha)_k/2^k(\alpha)_k>0$ and $c_{k-1}=\beta_{k+1}(2\alpha)_{k-1}/2^{k-1}(\alpha)_{k-1}>0$ (recall $\beta_{k+1}\in (\alpha_k,1)$, where $\alpha_k$ is the largest root of the Gegenbauer polynomials $P_k^{(n)}$).

The remaining of the proof follows the approach from \cite[Proposition 3.2]{CK}. Let $i\leq k-1$. The orthogonality implies that
\[ \int_{-1}^1 \left( P_{k+1}^{(n)}(t)+bP_{k-1}^{(n)}(t) \right)  \frac{P_i^{(n)}(t)-P_i^{(n)}(\beta_{k+1})}{t-\beta_{k+1}}\, d\mu_n(t)=0.\]
Thus,
\[ \int_{-1}^1 \frac{P_{k+1}^{(n)}(t)+bP_{k-1}^{(n)}(t)}{t-\beta_{k+1}}   P_i^{(n)}(t) \, d\mu_n(t)=  P_i^{(n)}(\beta_{k+1}) \int_{-1}^1 \frac{P_{k+1}^{(n)}(t)+bP_{k-1}^{(n)}(t)}{t-\beta_{k+1}} \, d\mu_n(t),\]
or, equivalently
\[  \int_{-1}^1 g_k(t)   P_i^{(n)}(t) \, d\mu_n(t)=  P_i^{(n)}(\beta_{k+1}) \int_{-1}^1 g_k(t) \, d\mu_n(t).\]
Utilizing the expansion \eqref{GegPosDef} we derive that 
\[ c_i\|P_i^{(n)}\|^2=P_i^{(n)}(\beta_{k+1})c_0, \quad i=0,1,\dots,k-1.\]
Since all of the roots of $P_i^{(n)}(t)$, $i=0,1,\dots,k-1$, are smaller than $\alpha_k<\beta_{k+1}$, we derive that $P_i^{(n)}(\beta_{k+1})>0$, and therefore all of the coefficients $c_0,c_1,\dots,c_{k-1}$ have the same sign. As $c_{k-1}>0$, the Lemma follows.
\end{proof}

\begin{theorem}\label{P_T PULB-space} The subspace $\mathcal{P}_{T^k}$ defined in \eqref{Skip1Add2} is a PULB-space. 
\end{theorem}

\begin{proof} 
The existence of a quadrature with positive weights required in Definition \ref{def-PULB-space} (i) has already been established in Lemma \ref{PositiveWeights}. We shall derive the needed for Definition \ref{def-PULB-space} (ii) interpolating polynomial next.

Let $h$ be any absolutely monotone potential. Denote with $H_{2k+1}(t;h)$ the Hermite interpolant in the space $\mathcal{P}_{2k+1}$ at the multi-set $\{t_1,t_2,\dots,t_{2k+2}\}=\{ \beta_1,\beta_1,\beta_2,\dots,\beta_{k+1},\beta_{k+1}\}$. Let $u_j(t):=(t-t_1)\dots(t-t_j)$, $j=1,\dots,2k+2$, denote the partial products associated with the multi-set. Using the Newton's formula we write
\begin{align*} 
H_{2k+1}(t;h)&=\sum_{j=0}^{2k+1}h[t_1,\dots,t_{j+1}]u_j(t)\\
&=\sum_{j=0}^{2k-1}h[t_1,\dots,t_{j+1}]u_j(t)+h[t_1,\dots,t_{2k+1}]g_k^2(t)+h[t_1,\dots,t_{2k+2}]g_k(t)g_{k+1}(t),\end{align*}
where $g_{k+1}(t)=(t-\beta_1)\dots(t-\beta_{k+1})=(2\alpha)_{k+1}/2^{k+1}(\alpha)_{k+1}\left( P_{k+1}^{(n)}(t)+bP_{k-1}^{(n)}(t)\right)$ clearly has positive Gegenbauer coefficients. From Lemma \ref{PosDef}, $g_k$ also has positive Gegenbauer coefficients, so from \cite{G} we have that $g_k^2$, $g_kg_{k+1}$, and $g_{k+1}^2$ all have positive Gegenbauer coefficients and in particular their $2k$-th coefficients will be positive. But this implies the $2k$-th Gegenbauer coefficient $h_{2k}$ of the interpolant $H_{2k+1}(t;h)$ is also positive. Let $e_{2k}$ denote the $2k$-th Gegenbauer coefficient of $g_{k+1}^2$ and consider the polynomial
\[ G_{2k+2}(t;h,\mathcal{P}_{T^{k}}):=H_{2k+1}(t;h)- \frac{h_{2k}}{e_{2k}}g_{k+1}^2(t).\]
It clearly interpolates $H_{2k+1}(t;h)$ and hence $h(t)$ and its derivative at the nodes $\{\beta_i\}$. It also is orthogonal to $P_{2k}^{(n)}$ and hence belongs to $\mathcal{P}_{T^k}$. Finally, from the absolute monotonicity of $h$ the regular interpolant $H_{2k+1}(t;h)\leq h(t)$ on $[-1,1]$ (follows easily by an application of Rolle's theorem), so we have
\begin{equation} \label{Interpolant_P_T}
G_{2k+2}(t;h,\mathcal{P}_{T^{k}})\leq H_{2k+1}(t;h) \leq h(t),
\end{equation}
verifying Definition \ref{def-PULB-space} (ii). The theorem is proved.
\end{proof}

\subsection{Second level PULB for $T^k$-designs}

Next, equipped with Theorem \ref{P_T PULB-space}, we proceed with the PULB for $T^k$-designs which we call {\it second level PULB}. 
\begin{theorem} \label{level-2-T}
Suppose $C$ is a spherical $T^k$-design of cardinality $N$ on $\mathbb{S}^{n-1}$, 
the potential $h$ is absolutely monotone on $[-1,1]$. Then 
\begin{equation} \label{level-2-bound}
\mathcal{Q}_h(C) \geq N \sum_{i=1}^{k+1} \gamma_i h(\beta_i).
\end{equation}
Equality in \eqref{level-2-bound} holds if and only if there exists a point $\widetilde{x} \in \mathbb{S}^{n-1}$ such that 
$I(\widetilde{x}, C) =\{\beta_i\}_{i=1}^{k+1}$ with integer frequencies $\{N\gamma_i\}_{i=1}^{k+1}$. \textcolor{black}{The location of such a point is independent of the potential $h$ and hence will be referred to as a universal minimum.}
\end{theorem}

\begin{remark} In the examples of Section 5, the absolute monotonicity requirement for $h$ may be relaxed to $h^{(2k)}\geq 0$, $h^{(2k+1)}\geq 0$, and $h^{(2k+2)}\geq 0$.
\end{remark}

\begin{proof} We already established that the roots of $q_{k+1}$ define a quadrature rule that is exact on  $\mathcal{P}_{T^k}$. The interpolant $G_{2k+2}(t;h,\mathcal{P}_{T^k})$ defined in \eqref{Interpolant_P_T} belongs to $\mathcal{L}(n,T^k,h)$ and agrees with $h$ at the quadrature nodes $ \{\beta_i\}_{i=1}^{k+1}$. Now the
second level PULB \eqref{level-2-bound} follows from Corollary \ref{cor4-5}. Should a point $\widetilde{x}$ exist, for which equality in \eqref{level-2-bound} holds, the independence of its location from the potential $h$ follows from that of the nodes $\beta_i$ and weights $\gamma_i$. \end{proof}

In the next section we will verify directly that for the icosahedron, dodecahedron, and the roots of $E_8$ and the 
Leech lattices, the corresponding polynomial subspaces $\mathcal{P}_{T^k}$ are PULB-spaces, thus providing 
concrete applications of the general result of Theorem \ref{P_T PULB-space}. We will also determine the universal 
minima $\widetilde{x}$ for which equality holds in \eqref{level-2-bound}. The concept of a PULB space and its 
application to a third-level PULB is found in Theorem \ref{600}, which considers the universal minima for the 
$600$-cell for potentials satisfying the conditions in case (ii) of Theorem \ref{PULB}. The considerably more 
difficult case (i) requires even higher level PULB's and will be the focus of a future work, along with the infinite 
family of sharp codes in the last row of Table 2.

\section{Next level PULB -- icosahedron, dodecahedron, $E_8$, and Leech lattice}  \label{sec-level2}

In this section we illustrate Theorem \ref{level-2-T} and the next level PULB \eqref{level-2-bound}. Three of the Platonic solids -- tetrahedron, cube, and  octahedron -- attain the first level PULB \eqref{PolarizationULB} (see \cite{BDHSS_P} for details), while the icosahedron and the dodecahedron attain the second level PULB \eqref{level-2-bound}. Similarly, with the exception of the infinite family (to be considered in a future work), all sharp codes in Table 2 attain \eqref{PolarizationULB}, with the exception of the icosahedron, and the kissing configurations of $E_8$ and the Leech lattice, which attain the second level PULB \eqref{level-2-bound}. In a recent independent work \cite{Bor-new} Borodachov analyzed via different methods the icosahedron, dodecahedron and the kissing configuration of  $E_8$.

\begin{theorem} \label{sharp-pos-level-2}
The sharp codes from Table 3 attain the bound \eqref{level-2-bound} with inner products and distance distributions  as given in that table 
and points $\widetilde{x}$ as described below. 
\end{theorem}


\subsection{Icosahedron} We remind the reader that the icosahedron, denoted here as $I_{12}$, is a sharp code with inner products $-1, \pm 1/\sqrt{5}$, has $12$ vertices, $20$ equilateral triangles as faces, and $30$ equal edges. It is also true that the icosahedron solves the best covering problem for $12$ points on $\mathbb{S}^2$.  We shall show that the minima of the discrete potential are attained at the centers of the spherical caps determined by the equilateral triangle faces. 

Let $\{x_1,x_2,x_3\}$ form one such side. Orient it horizontally, so that the center of the circumscribed spherical cap $\widetilde{x}$ is located at the North pole. We have that $\widetilde{x}=(x_1+x_2+x_3)/\|x_1+x_2+x_3\|$, from which we can find that
\[\widetilde{x}\cdot x_i =\frac{\|x_1+x_2+x_3\|}{3}=\frac{\sqrt{3+2x_1\cdot x_2+2x_1 \cdot x_3+2x_2 \cdot x_3}}{3}=\frac{\sqrt{1+2/\sqrt{5}}}{\sqrt{3}}, \quad i=1,2,3.\]
If we order the rest of the points of the icosahedron into three horizontal planes we find that $I(\widetilde{x},I_{12})=\{b_1,b_2,b_3,b_4\}$, where
\begin{equation}\label{Icos_b}
b_1=-\frac{\sqrt{1+2/\sqrt{5}}}{\sqrt{3}},\quad b_2= -\frac{\sqrt{1-2/\sqrt{5}}}{\sqrt{3}}, \quad b_3=\frac{\sqrt{1-2/\sqrt{5}}}{\sqrt{3}}, \quad b_4=\frac{\sqrt{1+2/\sqrt{5}}}{\sqrt{3}}.
\end{equation}

Let us compare this result to what our skip 1-add 2 method produces. In this case ($k=3, n=3$) the system \eqref{System} gives
$a=d=0$ and $b=5/9$, $c=-1$. Recalling \eqref{q_k+1}, we see that the polynomial $q_4(t)$ is 
\[ q_4(t)=P_{4}^{(3)}(t)+bP_{2}^{(3)}(t)=\frac{35}{8}t^4-\frac{35}{12}t^2+\frac{7}{72} \]
and its zeros are found to be exactly $b_1,b_2, b_3, b_4$.

Since $M_i(I_{12})=0$ for all odd $i$ and for $i=2,4$, and $8$, we set $T^3:=\{1,2,3,4,5,7,8\}$ and consider (see \eqref{Skip1Add2})
the subspace
\[ \mathcal{P}_{T^3}:=\mathcal{P}_{5} \oplus  \mbox{span} \left(P_{7}^{(3)}, P_{8}^{(3)}\right)=\mathcal{P}_8\cap \{P_6^{(3)}\}^\perp. \]
Theorem \ref{P_T PULB-space} gives that $\mathcal{P}_{T^3}$ is a PULB-space and hence the second level bound \eqref{level-2-bound}
from Theorem \ref{level-2-T} holds. The optimal polynomial $G_{8}(t;h,\mathcal{P}_{T^{3}})$, is found as follows. If $u_j(t)=(t-t_1)\dots(t-t_j)$, $j=0,\dots,8$ are the partial products associated with the multi-set $\{t_1,t_2,\dots,t_8\}=\{b_1,b_1,\dots,b_4\}$, then
\[G_{8}(t;h,\mathcal{P}_{T^{3}}) =\sum_{i=0}^7 h[t_1,\dots,t_{j+1}]u_j(t)- \frac{h_{6}}{e_{6}}u_8(t),\]
where 
\[ h_{6}=\frac{16}{231}h[t_1,\dots,t_7] +\frac{16\sqrt{75 + 30\sqrt{5}}}{3465} h[t_1,\dots,t_8] ,\quad e_6 = \frac{128}{3465}.\] 
If $h^{(6)}(t)\geq 0$, $h^{(7)}(t)\geq 0$, and $h^{(8)}(t)\geq 0$ on $[-1,1]$, then $G_{8}(t;h,\mathcal{P}_{T^{3}}) \leq h(t)$ on $[-1,1]$ and $G_{8}(t;h,\mathcal{P}_{T^{3}})\in \mathcal{L}(3,T^3,h)$. The quadrature rule 
\begin{equation} \label{quad-icosa-2}
\frac{1}{2}\int_{-1}^1 f(t)dt=\frac{1}{4}f(b_1)+\frac{1}{4}f(b_2)+\frac{1}{4}f(b_3)+\frac{1}{4}f(b_4)
\end{equation}
holds for all polynomials in $\mathcal{P}_{T^3}$, in particular for $G_{8}(t;h,\mathcal{P}_{T^{3}})$. Thus,
\[ 12G_0=3\sum_{i=1}^4 G_{8}(b_i;h,\mathcal{P}_{T^{3}})= 3\sum_{i=1}^4 h(b_i)=U_h(\widetilde{x},I_{12}),\]
because $G_{8}(b_i;h,\mathcal{P}_{T^{3}})=h(b_i)$. This implies that $U_h(x,I_{12})$ attains its global minimum at $\widetilde{x}$ (as well as at all other centers of circumscribed spherical caps associated with the equilateral triangles). In other words, the icosahedron attains the second level bound 
\eqref{level-2-bound} from Theorem \ref{level-2-T}.

We have the lower bound
\begin{equation} \label{IcosLower}
\mathcal{Q}_{h}(3,12)\geq \mathcal{Q}_h(I_{12})= 3\sum_{i=1}^4 h(b_i),
\end{equation}
where $\{ b_i \}$ are given in \eqref{Icos_b}. Moreover, as a by-product we obtain that $b_4$, the largest node in the quadrature 
\eqref{quad-icosa-2}, is a lower bound for the covering radius of all spherical 5-designs on $\mathbb{S}^2$, an improvement of the Fazekas-Levenshtein bound, 
that is attained by the icosahedron.

\subsection{Dodecahedron} We can utilize the above framework to find a second level PULB as well as 
a similar generalization of the Fazekas-Levenshtein bound, for the case of spherical $5$-designs of $20$ points on $\mathbb{S}^2$. 
While not a sharp code, the dodecahedron $D_{20}$ is one of the Platonic solids and of independent interest. We include it also because it highlights an important phenomenon, namely that the annihilating polynomial $q_{k+1}(t)$ does not depend on the cardinality $N$, but only on $k$ and $n$. 
As such, it is the same as for the icosahedron. These properties follow from the duality of the icosahedron and dodecahedron. 

Recall that the dodecahedron is an antipodal spherical 5-design, has $20$ vertices, $12$ regular pentagons as its faces and $30$ equal edges, and,
moreover, its eighth moment is zero. Let $\widetilde{x}$ be the center of the spherical cap circumscribing one of the regular pentagon faces. Orienting $\widetilde{x}$ as the North pole and the associated pentagon to be horizontal, we observe that the vertices are partitioned into four parallels. The inner products in $I(\widetilde{x},D_{20})$ are the same as in the case of the icosahedron, namely the numbers $b_1,b_2,b_3,b_4$ from \eqref{Icos_b},
but now with multiplicities $5$ each. Observe, that the polynomial $q_{4}(t)$ in \eqref{q_k+1}  depends on the design strength $\tau$ and dimension $n$, but not on cardinality of the code, so it is the same in this case. Noticing that $I(\widetilde{x},D_{20})$ gives rise to the same partial products as in the icosahedron case, we deduce the second level PULB
\[\mathcal{Q}_{h}(3,20)\geq \mathcal{Q}_h(D_{20})= 5\sum_{i=1}^4 h(b_i). \]


\begin{center}
\begin{table}
\scalebox{0.7}{
\begin{tabular}{|c|c|c|c|c|c|c|}
\hline
dim & Cardinality & Strength & Second level Polarization bound \\ 
$n$ & $N$ & $\tau$ & $h^{(\tau+1)} \geq 0,h^{(\tau+2)} \geq 0,h^{(\tau+3)} \geq 0$ \\
\noalign{\smallskip}\hline\noalign{\smallskip}
3 & 12 & 5 & $3h\left(-\frac{\sqrt{1+2/\sqrt{5}}}{\sqrt{3}}\right) + 3h\left(-\frac{\sqrt{1-2/\sqrt{5}}}{\sqrt{3}}\right) + 
3h\left(\frac{\sqrt{1-2/\sqrt{5}}}{\sqrt{3}}\right) + 3h\left(\frac{\sqrt{1+2/\sqrt{5}}}{\sqrt{3}}\right)$ \\
\noalign{\smallskip}\hline\noalign{\smallskip}
3 & 20 & 5 & $5h\left(-\frac{\sqrt{1+2/\sqrt{5}}}{\sqrt{3}}\right) + 5h\left(-\frac{\sqrt{1-2/\sqrt{5}}}{\sqrt{3}}\right) + 
5h\left(\frac{\sqrt{1-2/\sqrt{5}}}{\sqrt{3}}\right) + 5h\left(\frac{\sqrt{1+2/\sqrt{5}}}{\sqrt{3}}\right)$ \\ 
\noalign{\smallskip}\hline\noalign{\smallskip}
8 & 240 & 7 & $14h\left(-\frac{\sqrt{2}}{2}\right) + 64h\left(-\frac{\sqrt{2}}{4}\right) + 84h(0) + 64h\left(\frac{\sqrt{2}}{4}\right) + 
14h\left(\frac{\sqrt{2}}{2}\right)$ \\
\noalign{\smallskip}\hline\noalign{\smallskip}
24 & 196560 & 11 & $552h\left(-\frac{\sqrt{6}}{4}\right) +11178h\left(-\frac{\sqrt{6}}{6}\right)+48600 h\left(-\frac{\sqrt{6}}{12}\right)+75900h(0)+48600 h\left(\frac{\sqrt{6}}{12}\right)+11178h\left(\frac{\sqrt{6}}{6}\right)+552h\left(\frac{\sqrt{6}}{4}\right)$ \\
   &      &      &  \\
\hline
\end{tabular}
}
\bigskip
\caption{Second-level PULB quadrature for the icosahedron, dodecahedron, $E_8$, and Leech lattice.}
\end{table}
\end{center}

%

\subsection{The sharp code $(8,240,7)$ and the $E_8$ lattice} The last two codes we will consider, the kissing configurations of 
the $E_8$ lattice and the Leech lattice, have received significant attention lately with the discovery that these lattices solve the hard spheres best packing problems in the relevant dimensions (see \cite{V}, \cite{CKMRV}). The points of contact of a fixed sphere with the other neighboring spheres are $240$ 
(see Figure 3) in $\mathbb{R}^8$ and $196560$ in $\mathbb{R}^{24}$ and 
define sharp codes that do not attain the first level PULB \eqref{PolarizationULB}. We shall see that like the icosahedron, 
both attain the second level PULB  \eqref{level-2-bound} as described in Section \ref{T-designs} and shown in Table 3.

\begin{figure}[htbp]
\centering
\includegraphics[width=3 in]{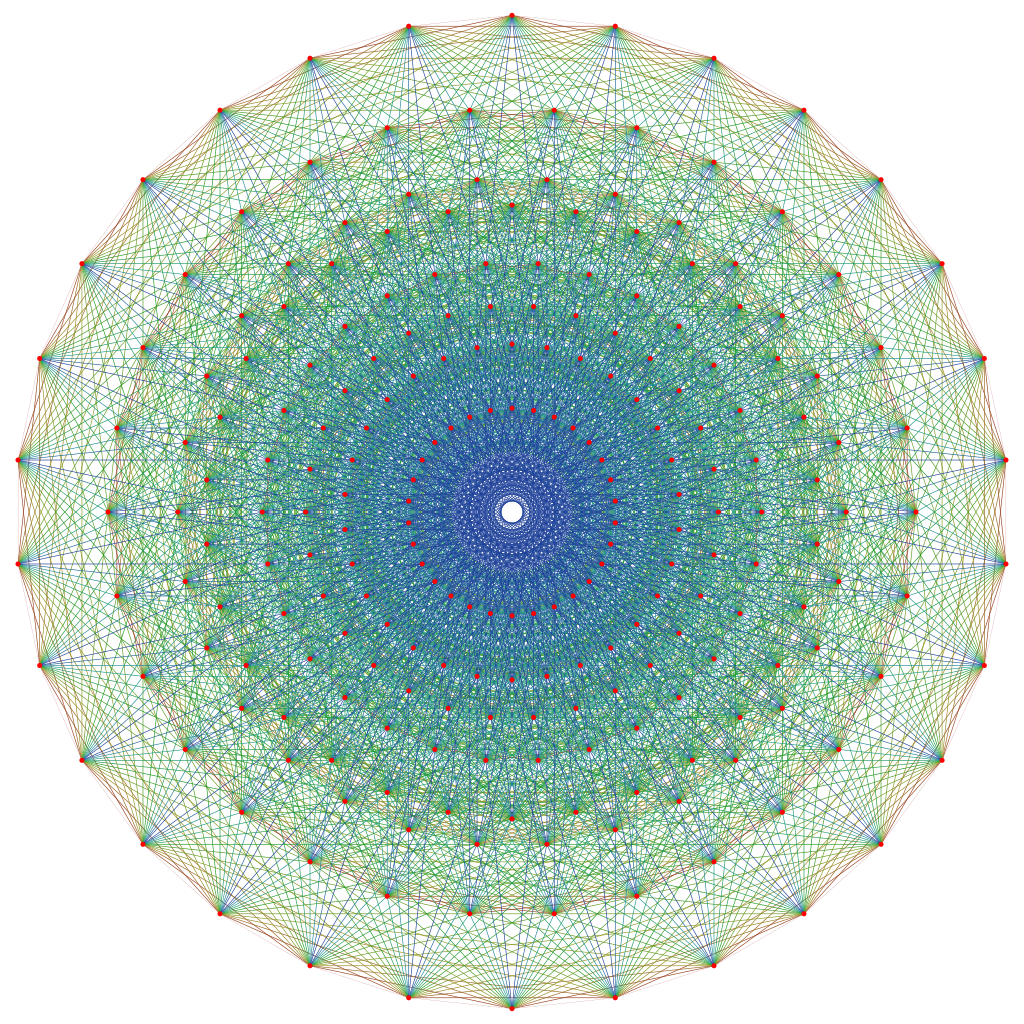}
\caption{The roots of $E_8$ lattice (By Claudio Rocchini - Own work, CC BY 3.0, https://commons.wikimedia.org/w/index.php?curid=4932406)}
\label{fig:3}
\end{figure}

We first exhibit the coordinates of the sharp code $C_{240}:=(8,240,7)$. There are $2^7=128$ vectors $[(\pm 1/\sqrt{8})^8]$ with 
even number of negative signs and $4 {8\choose 2}=112$ points with two coordinates $[(\pm 1/\sqrt{2})^2,0^6]$. 
As seen in Table 2, the quadrature weights for the $(8,240,7)$ sharp code are not integers and the code $C_{240}$ 
does not attain the PULB \eqref{PolarizationULB}. Thus, we proceed with the Skip 1-Add 2 framework as in Section \ref{T-designs}. 
 
In this case $k=4$ and $n=8$, so the system \eqref{System} gives $a=d=0$, $b=1/6$ and $c=-1$. The quadrature annihilating polynomial $q_5 (t)$ is
\[q_5(t)=\frac{t(8t^2-1)(2t^2-1)}{6}\]
with zeros 
\[ b_1= -\frac{\sqrt{2}}{2}, \quad b_2= -\frac{\sqrt{2}}{4}, \quad b_3=0, \quad b_4=\frac{\sqrt{2}}{4}, \quad b_5=\frac{\sqrt{2}}{2}.\]

Let us select $\widetilde{x}=[1,0,\dots,0]\in \mathbb{S}^7$. There are $14$ points with $\widetilde{x}\cdot [1/\sqrt{2},(\pm 1/\sqrt{2})^1,0^6]=1/\sqrt{2}$, $14$ with $\widetilde{x}\cdot [-1/\sqrt{2},(\pm 1/\sqrt{2})^1,0^6]=-1/\sqrt{2}$ and $84$ with $\widetilde{x}\cdot [0,(\pm 1/\sqrt{2})^2,0^5]=1/\sqrt{2}$. Similarly, there are $64$ points with $\widetilde{x}\cdot [1/\sqrt{8},(\pm 1/\sqrt{8})^7]=1/\sqrt{8}$, and $64$ points with $\widetilde{x}\cdot [-1/\sqrt{8},(\pm 1/\sqrt{8})^7]=-1/\sqrt{8}$, both with even number of negative signs. This data is shown in Table 3. 

Since $M_i(C_{240})=0$ for all odd $i$ and for $i=2,4,6$, and $10$, we set $T^4:=\{1,2,3,4,5,6,7,9,10\}$ and consider 
the subspace
\[ \mathcal{P}_{T^4}:=\mathcal{P}_{7} \oplus  \mbox{span} \left(P_{9}^{(8)}, P_{10}^{(8)}\right)=\mathcal{P}_{10} \cap \{P_8^{(8)}\}^\perp . \]
It is a PULB-space by Theorem \ref{P_T PULB-space} and the second level bound \eqref{level-2-bound} from Theorem \ref{level-2-T} holds. The optimal polynomial $G_{10}(t;h,\mathcal{P}_{T^4})$, is found as follows. If $u_j(t)=(t-t_1)\dots(t-t_j)$, $j=0,\dots,10$ are the partial products associated with the multi-set $\{t_1,t_2,\dots,t_{10}\}=\{b_1,b_1,\dots,b_5\}$, then
\[G_{10}(t;h,\mathcal{P}_{T^{4}}) =\sum_{i=0}^9 h[t_1,\dots,t_{j+1}]u_j(t)- \frac{h_{8}}{e_{8}}u_{10}(t),\]
where 
\[ h_{8}=\frac{143}{1280}h[t_1,\dots,t_9] +\frac{143\sqrt{2}}{2560} h[t_1,\dots,t_{10}] ,\quad e_8 = \frac{143}{2048}.\] 
If $h^{(8)}(t)\geq 0$, $h^{(9)}(t)\geq 0$, and $h^{(10)}(t)\geq 0$ on $[-1,1]$, then $G_{10}(t;h,\mathcal{P}_{T^{4}}) \leq h(t)$ on $[-1,1]$ and $G_{10}(t;h,\mathcal{P}_{T^{4}})\in \mathcal{L}(8,T^4,h)$. 
The quadrature rule 
\[240 \int_{-1}^1 f(t)d\mu_8 (t)=14f(b_1)+64f(b_2)+84f(b_3)+64f(b_4)+14f(b_5)\]
holds for all polynomials in $\mathcal{P}_{T^4}$, and hence for $G_{10}(t;h,\mathcal{P}_{T^{4}}) \in \mathcal{L}(8,T^4,h)$. Thus, as in the icosahedron case
\[ 240G_0=U_h(\widetilde{x},C_{240}),\]
because $G_{10}(b_i;h,\mathcal{P}_{T^{4}})=h(b_i)$. This implies that $U_h(x,C_{240})$ 
attains its global minimum at $\widetilde{x}$ and the sharp code $C_{240}$
with the so chosen $\widetilde{x}$ attains the second level bound \eqref{level-2-bound}. The classification of the points of minima will be given 
in Section 6.

\subsection{The sharp code $(24,196560,11)$ and the Leech lattice} The most complex known sharp code is the one arising from the shortest vectors of the Leech lattice. Recall that $C_L=(24,196560,11)$ was introduced in Subsection \ref{4600} and the three types of shortest vectors 
in the Leech lattice to be normalized in order to belong to $\mathbb{S}^{23}$ are $1104$ points of type A 
$[(\pm 1/\sqrt{2})^2,0^{22}]$; $98304$ points of type B $[(\mp 3/\sqrt{32})^1,(\pm 1/\sqrt{32})^{23}]$, where the upper signs follow the $1$'s in the codewords of the extended binary Golay code; and $97152$ of type C  $[(\pm 1/(2\sqrt{2}))^8,0^{16}]$. The PULB quadrature from Table 2 does not have integer weights, so we proceed with the Skip 1-Add 2 process. With $k=6$ and $n=24$, the system \eqref{System} implies 
$a=d=0$, $b=4/81$ and $c=-1$. The quadrature annihilating polynomial $q_7 (t)$ is
\[q_7(t)=\frac{17t(6t^2-1)(24t^2-1)(8t^2-3)}{9315}\]
with zeros 
\[ b_1= -\frac{\sqrt{6}}{4}, \quad b_2= -\frac{\sqrt{6}}{6}, \quad b_3=-\frac{\sqrt{6}}{12}, \quad b_4=0, \quad b_5=\frac{\sqrt{6}}{12}, \quad b_6=\frac{\sqrt{6}}{6} \quad b_7=\frac{\sqrt{6}}{4}.\]
The corresponding quadrature weights can be found from the Lagrange basis corresponding to the node set $\{ b_1,\dots, b_7\}$ (see the last row of 
Table 3; also below). 

We set $T^6:=\{1,2,\ldots,14\} \setminus \{12\}$ and consider the corresponding polynomial subspace  
\[ \mathcal{P}_{T^6}= \mathcal{P}_{14}\cap \{P_{12}^{(24)}\}^\perp. \]
Then $\mathcal{P}_{T^6}$ is a PULB-space by Theorem \ref{P_T PULB-space} and the second level bound \eqref{level-2-bound} from Theorem \ref{level-2-T} holds. The optimal polynomial $G_{14}(t;h,\mathcal{P}_{T^6})$, is found as follows. If $u_j(t)=(t-t_1)\dots(t-t_j)$, $j=0,\dots,14$ are the partial products associated with the multi-set $\{t_1,t_2,\dots,t_{14}\}=\{b_1,b_1,\dots,b_7\}$, then
\[G_{14}(t;h,\mathcal{P}_{T^{6}}) =\sum_{i=0}^{13} h[t_1,\dots,t_{j+1}]u_j(t)- \frac{h_{12}}{e_{12}}u_{14}(t),\]
where 
\[ h_{12}=\frac{310155}{2315264} h[t_1,\dots,t_{13}] +\frac{310155\sqrt{6}}{9261056} h[t_1,\dots,t_{14}] ,\quad e_{12} = \frac{516925}{5292032}.\] 
If $h^{(12)}(t)\geq 0$, $h^{(13)}(t)\geq 0$, and $h^{(14)}(t)\geq 0$ on $[-1,1]$, then $G_{14}(t;h,\mathcal{P}_{T^{6}}) \leq h(t)$ on $[-1,1]$ and $G_{14}(t;h,\mathcal{P}_{T^{6}})\in \mathcal{L}(24,T^6,h)$. For any polynomial $f\in \mathcal{P}_{T^6}$ we have
\begin{equation}\label{LeechQR}
\begin{split}
196560\int_{-1}^1 f(t)\, d\mu_{24}(t)=&552f\left(-\frac{\sqrt{6}}{4}\right) +11178f\left(-\frac{\sqrt{6}}{6}\right)+ 
48600 f\left(-\frac{\sqrt{6}}{12}\right)+75900f(0)\\
&+48600 f\left(\frac{\sqrt{6}}{12}\right)+11178f\left(\frac{\sqrt{6}}{6}\right)+552f\left(\frac{\sqrt{6}}{4}\right)
\end{split}
\end{equation}
(see the last row of Table 3). 

Let us select the point $\widetilde{x}=[5/\sqrt{48},1/\sqrt{48},\dots,1/\sqrt{48}]\in \mathbb{S}^{23}$. We shall prove that 
\[I(\widetilde{x},C_L)=\left\{\pm \frac{\sqrt{6}}{4}, \pm \frac{\sqrt{6}}{6},\pm \frac{\sqrt{6}}{12},0\right\}.\]
It is easy to see that the inner products of $\widetilde{x}$ with points of type A yield values
\[ \left\{\pm \frac{5+1}{4\sqrt{6}},\pm \frac{5-1}{4\sqrt{6}}, \pm \frac{1+1}{4\sqrt{6}},\pm \frac{1-1}{4\sqrt{6}}\right\} = \left\{ \pm \frac{\sqrt{6}}{4}, \pm \frac{\sqrt{6}}{6}, \pm \frac{\sqrt{6}}{12},0 \right\}.\]

Next, determine the set of inner products of $\widetilde{x}$ with points of type B. Points starting with $3/\sqrt{32}$ will have $8$, $12$ or $16$ positive coordinates follow, which yields inner products with $\widetilde{x}$ of $\sqrt{6}/12$, $\sqrt{6}/6$ and $\sqrt{6}/4$, respectively. For points that start with $-3/\sqrt{32}$ the number of positive coordinates changes to $7$, $11$, and $15$ and the inner products are $-\sqrt{6}/4$, $-\sqrt{6}/6$, and $-\sqrt{6}/12$, respectively. If a point of type B starts with $1/\sqrt{32}$ and another coordinate is $3/\sqrt{32}$, then the number of other positive coordinates is $7$, $11$, and $15$ with inner products $0$, $\sqrt{6}/12$, and $\sqrt{6}/6$, respectively. If that other coordinate is $-3/\sqrt{32}$, then the number of positive coordinates changes to $6$, $10$, and $14$ and the inner products are respectively $-\sqrt{6}/12$, $0$, and $\sqrt{6}/12$. The case of a point of type $B$ starting with $-1/\sqrt{32}$ is handled analogously.

Points of type C may start with $1/(2\sqrt{2})$, $-1/(2\sqrt{2})$, or $0$. In the first case, the number of negative coordinates is $0$, $2$, $4$, or $6$ with respective inner products $\sqrt{6}/4$, $\sqrt{6}/6$, $\sqrt{6}/12$, and $0$. In the second case, the number of subsequent negative coordinates may be $1$, $3$, $5$, and $7$ yielding corresponding inner products of $0$, $-\sqrt{6}/12$, $-\sqrt{6}/6$, and $-\sqrt{6}/4$. Finally, if a point of type C has a first coordinate $0$, then there may be $0$, $2$, $4$, $6$, and $8$ subsequent negative coordinates, in which case the inner products are $\sqrt{6}/6$, $\sqrt{6}/12$, $0$, $-\sqrt{6}/12$, and $-\sqrt{6}/6$, respectively.

The code $C_L$ and the point $\widetilde{x}$ define a quadrature rule exact on $\mathcal{P}_{T^6}=\mathcal{P}_{14}\cap \{P_{12}^{(24)}\}^\perp$ because the corresponding moments are zero, see Lemma \ref{BHSLemma}. As the quadrature nodes are the same as the quadrature nodes of the rule listed on Table 3, the weights are determined uniquely from the Lagrange basis. Therefore, the frequencies of the inner products will match the corresponding ones in the quadrature rule on Table 3.

The polynomial $G_{14}(b_i;h,\mathcal{P}_{T^{6}}) \in \mathcal{L}(24,T^6,h)$ provides the second level PULB \eqref{level-2-bound}, 
attained by $C_L$ with the point $\widetilde{x}$ as described above; i.e.,
\[ \begin{split}
{\mathcal{Q}}_{h}(C_L)=&552h\left(-\frac{\sqrt{6}}{4}\right) +11178h\left(-\frac{\sqrt{6}}{6}\right)+ 
48600 h\left(-\frac{\sqrt{6}}{12}\right)+75900h(0)\\
&+48600 h\left(\frac{\sqrt{6}}{12}\right)+11178h\left(\frac{\sqrt{6}}{6}\right)+552h\left(\frac{\sqrt{6}}{4}\right).
\end{split} \]

\section{Classification of the universal minima of sharp codes when $h^{(\tau+1)}(t)\geq 0$}\label{classification}
In this section we shall characterize the universal minima of the sharp codes. The following theorem describes the minima of the discrete potential for codes that attain first and second level PULB. 

\begin{theorem} \label{6_1} Let $C\subset \mathbb{S}^{n-1}$, $|C|=N$, be a sharp code that attains the first or second level PULB, \eqref{PolarizationULB} or \eqref{level-2-bound}, respectively, as shown in Table 2 or 3. Then the global minimum of $U_h( x,C)$ is the corresponding 
polarization bound in Table 2 or Table 3. This minimum is attained at a point $y\in \mathbb{S}^{n-1}$ if and only if the collection of inner products $I(y, C)$, coincides with the set of quadrature nodes $\{\alpha_i \}_{i\in I}$ or $\{\beta_{i}\}_{i=1}^{k+1}$ 
and the frequencies of these inner products are $\{N\rho_i\}_{i\in I}$ or  $\{N\gamma_{i}\}_{i=1}^{k+1}$, respectively. 
Furthermore, $y$ is a center of the spherical cap determined by a hyperplane that is $\alpha_k$ or $\beta_{k+1}$, respectively, 
distance from the origin and is determined by the $N\rho_k$- or $N\gamma_{k+1}$-point facet of $C$ that is closest to $y$.  For all sharp codes and Platonic solids in Table 2 and 3 these facets are uniquely determined up to isometry and hence all minima are obtained from $\widetilde{x}$ by an isometric transformation.  
\end{theorem}

\begin{proof} This theorem is a direct consequence of the conditions in Theorems \ref{PULB} and \ref{level-2-T} for attaining the bounds \eqref{PolarizationULB} and \eqref{level-2-bound}, respectively. Indeed, if $C$ attains the  corresponding PULB, 
then the existence of a point $\widetilde{x}$ follows. Moreover, any other point $y$ where the global minimum is obtained will have to 
satisfy the same properties of that $\widetilde{x}$ does. In particular, $N\rho_k$ {or $N\gamma_{k+1}$ points, respectively} 
from the code will have to lie in a hyperplane $\{ z\in \mathbb{R}^n : y\cdot z=\alpha_k\}$  {or
$\{ z\in \mathbb{R}^n : y\cdot z=\beta_{k+1}\}$, respectively}. We enumerate below all sharp codes discussed above and the unique facets 
determining the minima, which completes the proof.
\end{proof}

We now describe the $N\rho_k$ or  $N\gamma_{k+1}$-point facets for the sharp codes in Table 2 and Table 3 explicitly.
\begin{itemize}
\item[1.] The both cases of the regular $N$-gon on $\mathbb{S}^1$, the quadrature determines two points of the $N$-gon circle that are closest to $y$ and hence $y$ is a midpoint of the arc between two consecutive  points from the $N$-gon.
\item[2.] Since the codes are degenerate in this case, we simply embed them in a Equatorial hyperplane and the North and South poles are the minima (for all possible embeddings).
\item[3.] For the regular simplex case $N=n+1$, the $n$-point facet is a simplex of one dimension lower and hence unique (all sharp codes are unique). 
\item[4.] For the cross-polytope code, the $n$-point facet in questions is also a simplex of one dimension lower and hence unique.
\item[5.] For the icosahedron we have to consider the second level PULB, where the facet is an equilateral triangle (see Table 3).
\item[6.] The $(5,16,3)$ sharp code has an $8$-point facet in a $4$-dimensional hyperplane at a distance   $1/\sqrt{5}$ from the origin. By \cite[Theorem 8.2]{DGS} we conclude the $8$-point facet forms a cross-polytope.
\item[7.] The $(6,27,4)$ sharp code has minima at points that are antipodal to points from the code (the facet is a $10$-point cross-polytope on a $5$-dimensional hyperplane at a distance $1/2$ from the origin).
\item[8.] The $(7,56,5)$ sharp code has a $12$-point facet on a $6$-dimensional hyperplane at a distance $1/\sqrt{3}$ from the origin. That it is a cross-polytope follows from \cite[Theorem 8.2]{DGS} again.
\item[9.] For the $E_8$ lattice code $(8,240,7)$ we need to consider the second-level PULB, which reveals that the corresponding $14$-point facet is a cross-polytope in a $7$-dimensional hyperplane at a distance $1/\sqrt{8}$ from the origin. 
\textcolor{black}{The description of these points of minima is given in \cite{Bor-new}; they are the vertices of the Gosset's regular polytope $4_{21}$ \cite{Gos}.}
\item[10.] The $(21,112,3)$ code has a $56$-point facet in a $20$-dimensional hyperplane at a distance $1/\sqrt{21}$ from the origin. While this is not a sharp code, it forms the Gewirtz graph, a unique strongly regular graph srg$(56,10,0,2)$ (see \cite[Subsection 10.20]{BrSRG}). 
\item[11.] The $(21,162,3)$ code has a $81$-point facet in a $20$-dimensional hyperplane, also at a distance $1/\sqrt{21}$ from the origin. This too is not a sharp code, but it forms the Brouwer-Haemers  graph \cite[Subsection 10.28]{BrSRG}, a unique strongly regular graph srg$(81,20,1,6)$. As there are two more strongly regular graphs with $81$ points,  the $VNO^-_4 (3)$ graph, a unique srg$(81,30,9,12)$ graph \cite[Subsection 10.29]{BrSRG}, and two rank $3$ conference graphs srg$(81,40,19,20)$ \cite[Subsection 10.30]{BrSRG}, we verified the distances within the $81$-point facet to conclude that the configuration is indeed Brouwer-Haemers graph.
\item[12.]  The Higman-Sims sharp code $(22,100,3)$ has a $50$-point facet that forms a Hoffman-Singleton graph, a unique strongly regular graph srg$(50,7,0,1)$ (see  \cite[Subsection 10.19]{BrSRG}). 
\item[13.] Similarly to the $(6,27,4)$ sharp code, the McLaughlin sharp code $(22,275,4)$ has minima at points that are antipodal to points from the code. The facet is a $112$-point sharp code (follows from \cite[Theorem 8.2]{DGS}).
\item[14.] The $(22,891,5)$ code has a $162$-point facet in a $21$-dimensional hyperplane at a distance $1/4$ from the origin. Utilizing \cite[Theorem 8.2]{DGS} we have that this facet forms a sharp code and hence is uniquely determined.
\item[15.]  The $(23,552,5)$ code has a $100$-point facet, forming a Higman-Sims sharp code embedded in a hyperplane that is $\sqrt{3}/5$ distance away from the origin.
\item[16.]  The $(23,4600,7)$ code has a $275$-point facet, forming a McLaughlin sharp code embedded in a hyperplane that is $\sqrt{5}/5$ distance away from the origin.
\item[17.] Finally, we conclude with the Leech lattice. The second level reveals that the facet closest to a minima is comprised of $552$-point code, which according to \cite[Theorem 8.2]{DGS} is a sharp code and hence is uniquely determined.
\end{itemize}

\section{Max-min polarization of sharp codes: $h^{(\tau+1)}(t)\leq 0$ case}\label{neg-derivative}

In this section we consider the max-min polarization of the sharp codes for the case of $h^{(\tau+1)}(t)\leq 0$. As Theorem \ref{PULB}, Case (ii) describes the quadrature nodes in this case are the zeros of the polynomials $(t-1)(t+1)^{1-\epsilon}P_{k-1+\epsilon}^{1,1-\epsilon}(t)$. Table 4 depicts the PULB quadratures for the sharp codes in this case. There are five codes, namely $(5,16,3)$, $(21,112,3)$, $(21,162,3)$, $(22,100,3)$ and $(22,891,5)$, along with the infinite family in the last row,
that do not attain the bound (their $N\rho_i$ are non-integer). Unlike the case of non-negative derivative case $h^{(\tau+1)}\geq 0$, 
where specialized analysis is needed for each code, here we are able to obtain that all other sharp codes attain the bound with a single theorem. 
Note that the presence of the inner product $1$ in the quadrature means that the optimality is attained with $x$ being any point of the code. 

We recall the definition of {\em strongly sharp codes} as sharp configurations that are even strength designs (see \cite[Definition 2.1]{B}), namely spherical $2k$-designs with $k$ distinct inner products, as only $k=1$ and $2$ are possible. We also shall refer to sharp configurations that 
are antipodal as {\em antipodal sharp codes}. The following theorem establishes the optimality of such codes in the case under consideration.

\begin{center}
\begin{table}
\scalebox{0.57}{
\begin{tabular}{|c|c|c|c|c|c|c|}
\hline
dim & Cardinality & Strength & Polarization bound \\ 
\noalign{\smallskip}\hline\noalign{\smallskip}
$n$ & $N$ & $\tau$ & $h^{(\tau+1)} \leq 0$ \\
\noalign{\smallskip}\hline\noalign{\smallskip}
$2$ & $N=2k$ & $2k-1$ & $h(-1) + 2\sum\limits_{j=1}^{k-1} h(\cos(2j\pi/N)) + h(1)$ \\ 
\noalign{\smallskip}\hline\noalign{\smallskip}
$2$ & $N=2k+1$ & $2k$ & $2\sum\limits_{j=1}^k h(\cos((2j-1)\pi/N))$ \\ 
\noalign{\smallskip}\hline\noalign{\smallskip}
$n$ & $N=2 \leq n$ & 1 & $(N/2)[h(-1) + h(1)]$ \\ 
\noalign{\smallskip}\hline\noalign{\smallskip}
$n$ & $n+1$ & 2 & $nh(-1/n) + h(1)$ \\ 
\noalign{\smallskip}\hline\noalign{\smallskip}
$n$ & $2n$ & 3 & $h(-1) + (2n-2)h(0) + h(1)$ \\ 
\noalign{\smallskip}\hline\noalign{\smallskip}
3 & 12 & 5 & $h(-1) + 5h(-1/\sqrt{5}) + 5h(1/\sqrt{5}) + h(1)$ \\ 
\noalign{\smallskip}\hline\noalign{\smallskip}
$5^{*}$ & 16 & 3 & $8[(1/5)h(-1) + (8/5)h(0) + (1/5)h(1)]$ \\ 
\noalign{\smallskip}\hline\noalign{\smallskip}
6 & 27 & 4 & $10h(-1/2) + 16h(1/4) + h(1)$ \\ 
\noalign{\smallskip}\hline\noalign{\smallskip}
7 & 56 & 5 & $h(-1) + 27h(-1/3) + 27h(1/3) + h(1)$ \\ 
\noalign{\smallskip}\hline\noalign{\smallskip}
8 & 240 & 7 & $h(-1) + 56h(-1/2) + 126h(0) + 56h(1/2) + h(1)$ \\ 
\noalign{\smallskip}\hline\noalign{\smallskip}
$21^{*}$ & 112 & 3 & $56[(1/21)h(-1) + (40/21)h(0) + (1/21)h(1)]$ \\ 
\noalign{\smallskip}\hline\noalign{\smallskip}
$21^{*}$ & 162 & 3 & $27[(1/7)h(-1) + (40/7)h(0) + (1/7)h(1)]$ \\ 
\noalign{\smallskip}\hline\noalign{\smallskip}
$22^{*}$ & 100 & 3 & $25[(1/11)h(-1) + (42/11)h(0) + (1/11)h(1)]$ \\ 
\noalign{\smallskip}\hline\noalign{\smallskip}
22 & 275 & 4 & $112h(-1/4) + 162h(1/6) + h(1)$ \\ 
\noalign{\smallskip}\hline\noalign{\smallskip}
$22^{*}$ & 891 & 5 & $81[(1/46)h(-1) + (126/23)h(-\sqrt{6}/12) + (126/23)h(-\sqrt{6}/12) + (1/46)h(1)]$ \\ 
\noalign{\smallskip}\hline\noalign{\smallskip}
23 & 552 & 5 & $h(-1) + 275h(-1/5) + 275h(1/5) + h(1)$ \\ 
\noalign{\smallskip}\hline\noalign{\smallskip}
23 & 4600 & 7 & $h(-1) + 891h(-1/3) + 2816h(0) + 891h(1/3) + h(1)$ \\ 
\noalign{\smallskip}\hline\noalign{\smallskip}
24 & 196560 & 11 & $h(-1) + 4600h(-1/2) + 47104h(-1/4) + 93150h(0) + 47104h(1/4) + 4600h(1/2) + h(1)$ \\
\noalign{\smallskip}\hline\noalign{\smallskip}
$q(q^3+1)/(q+1)^{*}$ & $(q^3+1)(q+1)$ & 3 & $(N/2)[(1/n)h(-1) + ((2n-2)/n)h(0) + (1/n)h(1)] $ \\
\noalign{\smallskip}\hline\noalign{\smallskip}
\end{tabular}
}
\bigskip
\caption{PULB Quadrature for sharp codes, non-positive case -- right-hand side of \eqref{PolarizationULB}. 
Codes that do not attain the PULB  \eqref{PolarizationULB} for Theorem \ref{PULB}, case (ii),  are indicated with *.}
\end{table}
\end{center}

\begin{theorem}\label{PULB_negative}
Let $h$ be a continuous potential on $[-1,1]$ with $h^{(\tau+1)}\leq 0$ on $[-1,1]$ and let $C$ be either a strongly sharp code or an antipodal sharp code. Then the minimum of $U_h(x,C)$ occurs at a point of the code $C$ and the code $C$ attains the PULB \eqref{PolarizationULB}.
\end{theorem}

\begin{proof}
Let $H_\tau (t;h)$ be the polynomial that interpolates the potential $h$ at the nodes $\alpha_i$, $i\in I$, and the derivative $h^\prime$ whenever $\alpha_i\in (-1,1)$. For antipodal sharp codes we have $\epsilon=0$. The Hermite interpolation error formula yields that
\[ f(t)-H_\tau (t;h)=\frac{h^{(\tau+1)}(\xi)}{(\tau+1)!}(t+1)(t-\alpha_1)^2\dots(t-\alpha_{k-1})^2(t-1) \geq 0, \quad t\in [-1,1]. \]
Similarly, for strongly sharp codes $\epsilon=1$ and the Hermite error formula is modified as follows
\[ f(t)-H_\tau (t;h)=\frac{h^{(\tau+1)}(\xi)}{(\tau+1)!}(t-\alpha_1)^2\dots(t-\alpha_{k})^2(t-1) \geq 0, \quad t\in [-1,1]. \]
Observe, that the distinct inner products between points of the code coincide with the quadrature nodes, which follows from \cite[Theorem 3.1]{BDHSS}. 
Clearly, any point of the code then can play the role of $\widetilde{x}$ in Theorem \ref{PULB}, which concludes the proof.
\end{proof}

In Sections \ref{T-designs} and \ref{sec-level2} we derived a generalization of the PULB \eqref{PolarizationULB} to $T$-designs and in particular 
a procedure Skip 1- Add 2 and what we refer to as second-level polarization. We proceed with a result for the 600-cell (Figure 4), 
that can be viewed (although the 600-cell is not a sharp code) through the same 
lens as a Skip 1-Add 4 and a third-level polarization. While we are able to determine the inner products between points
of the 600-cell and their frequency directly from the code, a third level ULB quadrature as in Table 1 may
be derived as in \cite[Section 5.2]{BDHSSMathComp}. In Theorem \ref{600} below we will see that the universal minimum is attained
at a point of the code, which leads to the universal polarization bound in the form shown in \eqref{600cell}.

For potentials continuous on $[-1,1]$ that have $h^{(i)}\geq 0$, $i=1,\dots,15$ and $h^{(16)}\leq 0$ we are able 
to adapt the proof of \cite[Theorem 5.1]{BDHSS_P} and derive the following result about the $600$-cell $C_{600}$. 
Recall that the inner product among points of the $600$-cell (no necessarily distinct) are given by
\[ B:=\left\{ -1, -\frac{1+\sqrt{5}}{4},-\frac{1}{2},\frac{1-\sqrt{5}}{4},0,\frac{\sqrt{5}-1}{4},\frac{1}{2}, \frac{1+\sqrt{5}}{4},1 \right\} =\{b_i\}_{i=1}^9.\]

\begin{figure}[htbp]
\centering
\includegraphics[width=3 in]{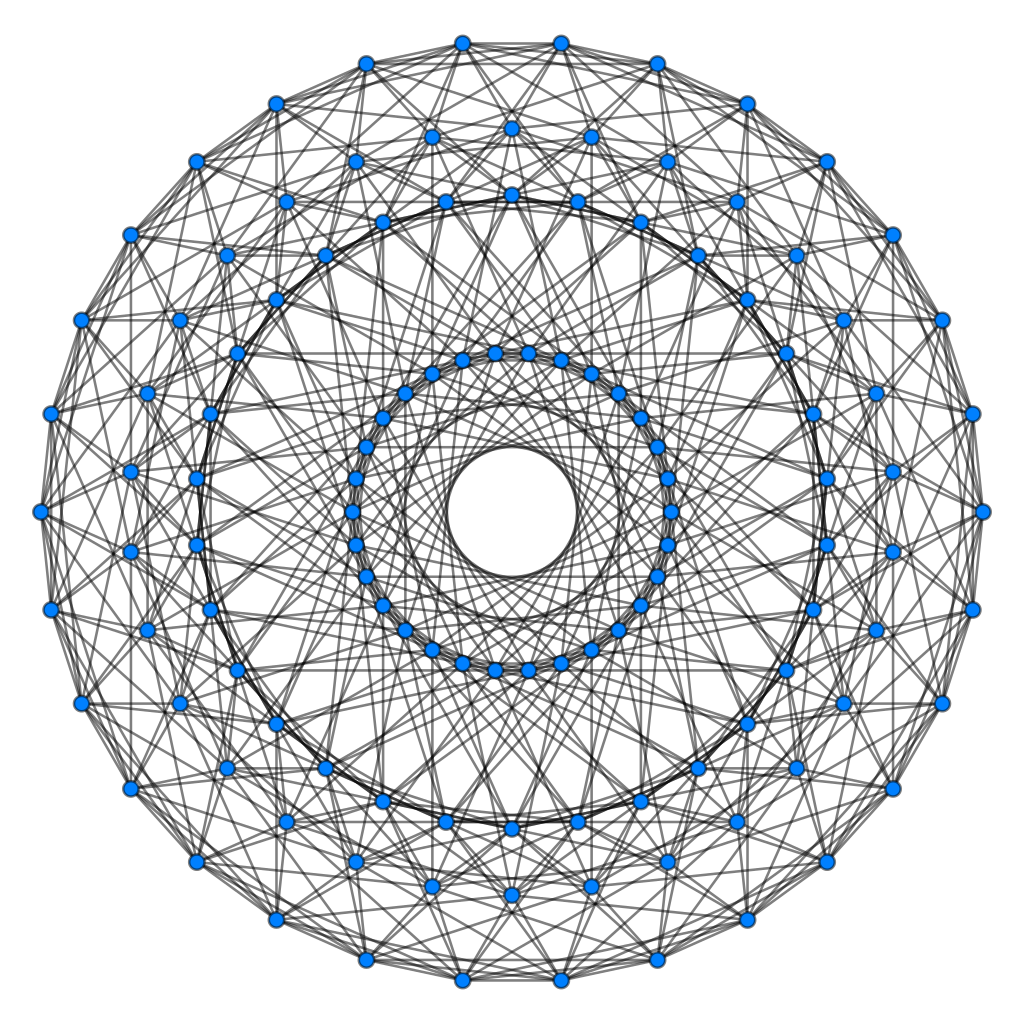}
\caption{ The $600$ cell (By Claudio Rocchini - Own work, CC BY 3.0, https://commons.wikimedia.org/w/index.php?curid=4481192)}
\label{fig:4}
\end{figure}

\begin{theorem}\label{600} 
Let $h$ be a continuous potential on $[-1,1]$ that has $h^{(i)}\geq 0$, $i=1,\dots,15$ and $h^{(16)}\leq 0$. Then the minimum of the discrete potential $U_h (x,C_{600})$ is attained at the points $y\in C_{600}$, i.e.
\begin{equation}\label{600cell} \mathcal{Q}_h (C_{600})=h(b_1)+12\sum_{j=1}^4h(b_{2j})+
20\sum_{i=0}^1 h(b_{4i+3})+30h(b_5) + h(b_9).\end{equation}
\end{theorem}

\begin{proof} The proof follows very closely the layout of the proof of Theorem 5.1 in \cite{BDHSS_P}. Introduce the multi-set 
\[I=\{ t_j \}_{j=1}^{16}:=\{b_1,b_2,b_2,b_3,b_3,\dots,b_8,b_8,b_9\}, \]
and consider the partial products 
\[ g_j(t):=(t-t_1)\cdots(t-t_j), \ j=1,\dots, 16,\ g_0(t):=1.\]
Denote with $g(t;h,I)$ the interpolating polynomial of $h$ at the nodes of $I$. The Newton interpolating formula yields that
\begin{equation}\label{NewtonFormula600cell} g(t;h,I)=\sum_{j=0}^{15} h[t_1,\dots,t_{j+1}]g_j(t) ,\end{equation}
where $h[t_1,\dots,t_{j+1}]$ denotes the divided difference of $h$ in the listed nodes.
Note that all of the divided differences of $h(t)$ in \eqref{NewtonFormula600cell} are nonnegative because 
$h[t_1,\dots,t_{j+1}]=h^{(j)}(\xi)/j!  \geq 0$. The Hermite interpolation error formula and $h^{(16)}\leq 0$ show that (recall that $b_9=1$)
\begin{equation}\label{HermiteError600cell}
h(t)-  g(t;h) = \frac{h^{(16)}(\xi)}{16!}(t-b_1)(t-b_2)^2 \dots (t-b_8)^2(t-b_9)\geq 0. \quad t\in [-1,1].
\nonumber
\end{equation}
In \cite[Theorem 5.1]{BDHSS_P} it was shown that the $12$-th Gegenbauer coefficients of $g(t;h,I)$ and $g_{16}$ are positive, so the polynomial
\[H(t;h,I):=g(t;h,I)-\frac{(g(t;h,I))_{12}}{(g_{16})_{12}} g_{16}(t) = \sum_{i=0}^{16} H_{i}P_i^{(4)}(t),\]
interpolates $h$ at the points of $I$, $H_{12}=0$, and $H(t;h,I)\leq h(t)$. This implies that
\begin{eqnarray*}
U_h(z,C_{600}) &=& \sum_{y\in C_0}h( z\cdot y)\geq \sum_{y\in C_0} H(z\cdot y;h) \\
&=& 120H_0=\sum_{y\in C_0} H(w\cdot y;h)= \sum_{y\in C_0}h( w\cdot y) \\
&=& h(b_1)+12\sum_{j=1}^4h(b_{2j})+20\sum_{i=0}^1 h(b_{4i+3})+30h(b_5) + h(b_9).
\end{eqnarray*}
That the only minima are points of the code follows analogously to \cite{BDHSS_P}.
\end{proof}

The minima of the discrete potential induced by the $600$-cell for absolutely monotone potentials (or $h^{(16)}\geq0$) is a much more complex problem that will be addressed in a subsequent work.

\noindent{\bf Acknowledgment.}  This material is based upon work supported by the National Science Foundation under Grant No. DMS-1929284 while the authors were in residence at the Institute for Computational and Experimental Research in Mathematics in Providence, RI, during the {\it Colaborate@ICERM} program.  The research of the first author was supported, in part, by Ministry of Education and Science of Bulgaria under Grant no. DO1-387/18.12.2020 “National Centre for High-Performance and Distributed Computing”. The research of the second author was supported, in part, by NSF grant DMS-1936543.
The research of the fifth author was supported, in part, by Bulgarian NSF grant KP-06-N32/2-2019. 

\end{document}